\documentclass[12pt,reqno]{amsart}
\usepackage{amsmath, color}
\usepackage[margin=1.5in]{geometry}
\usepackage{graphicx}
\usepackage{amsfonts}
\usepackage{amssymb}
\usepackage[utf8]{inputenc}
\usepackage{comment}
\usepackage{hyperref}
\usepackage{mathtools}
\usepackage{ dsfont }
\usepackage{esint}

\newtheorem{theorem}{Theorem}[section]
\theoremstyle{plain}

\newtheorem{lemma}{Lemma}[section]

\newtheorem{proposition}{Proposition}[section]

\numberwithin{equation}{section}
\theoremstyle{definition}

\theoremstyle{remark}
\newtheorem{remark}{Remark}[section]
\allowdisplaybreaks

\def\pd{\partial}
\def\re{\mathbb{R}}

\def\mbb{\mathbb}

\newcommand{\osc}{\mathrm{osc}}
 
\title[Doubling and the critical phase]{Doubling and the two-dimensional critical valued Lagrangian phase}

\author{Arunima Bhattacharya, Ravi Shankar, and Jeremy Wall}

\address{Department of Mathematics, Phillips Hall\\
 the University of North Carolina at Chapel Hill, NC }
\email{arunimab@unc.edu}

\address{Department of Mathematics, Fine Hall\\
Princeton University, Princeton, NJ}
\email{rs1838@princeton.edu}

\address{Department of Mathematics, Phillips Hall\\
 the University of North Carolina at Chapel Hill, NC }
\email{jwall2@unc.edu}

\begin{document}

\begin{abstract}
In this paper, we establish interior Hessian and gradient estimates for the two-dimensional Lagrangian mean curvature equation when the phase changes signs, provided the gradient of the phase vanishes along its zero set. At the critical phase in two dimensions, the Jacobi inequality degenerates, preventing the use of higher-dimensional methods to obtain Hessian estimates. To address this difficulty, we introduce a modified doubling technique that applies to degenerate Jacobi inequalities and yields interior estimates. 
\end{abstract}

\maketitle

\section{Introduction}
We study the two-dimensional Lagrangian mean curvature equation
\begin{equation}
    \arctan \lambda_{1}+\arctan \lambda_{2}=\Theta(x) \label{s}
\end{equation} 
where $\Theta$ has bounded second derivatives. Here, $\lambda_1$ and $\lambda_2$ are the eigenvalues of the Hessian $D^2u$, and the function $\Theta$ is called the Lagrangian phase or angle of the Lagrangian submanifold  $L=(x, Du(x))\subset \mathbb{C}^2$. As shown by Harvey-Lawson \cite[Proposition 2.17]{HL}, $\Theta$ acts as the potential of the mean curvature vector of $(x,Du(x))$: 
\begin{equation*}
    \vec H=J\nabla_g\Theta,
\end{equation*}
where $g=I_2+(D^2u)^2$ is the induced metric on $L$, and $J$ is the almost complex structure on $\mbb C^2$. Thus, equation \eqref{s} is the potential equation for prescribed \textit{Lagrangian mean curvature}. When the Lagrangian phase is constant, $u$ solves the \textit{special Lagrangian equation} of Harvey-Lawson \cite{HL}. In this case, $H=0$, and $L$ is a volume-minimizing Lagrangian submanifold. 

An equivalent form of \eqref{s} can be written using the symmetric polynomials:
\begin{equation}\label{tanpde}
    \tan(\Theta) = \frac{\sigma_1}{1 - \sigma_2}
\end{equation}
where $\sigma_1(D^2u) = \Delta u = \lambda_1+\lambda_2$ and $\sigma_2(D^2u) = \det D^2u = \lambda_1\lambda_2$.

Our main results are the following.
    \begin{theorem}\label{thm1}
        Let $u$ be a smooth solution of \eqref{s} in $B_R(0)\subset\re^2$, where $-\pi<\Theta < \pi$, $D\Theta=0$ on the level set $\{\Theta = 0\}$ and $\Theta\in C^2(B_R(0))$. Then the Hessian of $u$ admits the bound
        \[
        |D^2u(0)|\leq C(\|u\|_{C^1(B_R(0))},\|\Theta\|_{C^2(B_R(0))},R).
        \]
    \end{theorem}

    \begin{remark}
        Combining this with the gradient estimate in Lemma \ref{gradlem}, we obtain
        \[
        |D^2u(0)|\leq C(\osc_{B_R(0)}u,\|\Theta\|_{C^2(B_R(0))},R).
        \]
    \end{remark}

\begin{remark}
A special case of Theorem \ref{thm1} shows that if $u$ is a smooth solution of \eqref{s} in $B_R(0)\subset \mathbb R^2$, $0 \leq |\Theta| < \pi$, and $\Theta\in C^2(B_R(0))$, then the Hessian of $u$ admits the bound
\[
|D^2u(0)| \leq C\!\left( \|u\|_{C^1(B_R(0))}, \|\Theta\|_{C^2(B_R(0))}, R \right).
\]
\end{remark}

  \begin{remark}
       If $0\le \Theta<\pi$ is nonnegative and $C^2(B_R(0))$, we find that there exists a solution $u\in C^3(\Omega)\cap C^0(\bar{\Omega})$ to the Dirichlet problem 
     \begin{equation*}\label{dbvp_1}
        \begin{cases}
             \arctan\lambda_1 +\arctan\lambda_2 = \Theta(x) \text{ in } \Omega\\
            u = \phi \text{ on } \pd\Omega
        \end{cases}
        \end{equation*}
    where $\phi\in C^0(\pd \Omega)$.  Here $\Omega$ is a uniformly convex bounded domain in $\mathbb{R}^2.$  This follows from classical solvability for positive phases $\Theta\ge \delta$ in \cite{AB1}, combined with the above interior Hessian estimate, the interior gradient estimate in \cite{BMS}, and the interior $C^{2,\alpha}$ estimate in \cite[Section 7.1]{BSWY25}.
    \end{remark}

    The range of the Lagrangian phase determines the concavity of the arctangent operator in \eqref{s}. In dimension $n$, the phase is hypercritical whenever $|\Theta|\geq (n-1)\pi/2$. In this case, the potential $u$ is convex, and the arctangent operator is concave. The phase is supercritical whenever $|\Theta|\geq (n-2)\pi/2 + \delta$ for some fixed $\delta > 0$. In this case, the eigenvalues could have opposite signs, but the smaller magnitude eigenvalue is bounded by a constant depending on $\delta$ as shown by Yuan \cite{YY0}. In the supercritical range, the arctangent operator can be modified to a concave operator by exponentiating it, see \cite{CPX} and \cite{CW2}. The phase is critical and supercritical whenever $|\Theta|\geq (n-2)\pi/2$. The level set $\{\lambda\in\re^n: \sum \arctan\lambda_j = \Theta\}$ is convex only when $|\Theta|\geq (n-2)\pi/2$ \cite{YY0}. In the subcritical case, i.e., when $|\Theta|< (n-2)\frac{\pi}{2}$, counterexamples to Hessian estimates exist, and the Dirichlet boundary value problem is not classically solvable with smooth boundary data as shown by Nadirashvili-Vl\u{a}du\c{t} \cite{NV}, Wang-Yuan \cite{WangY}, and Mooney-Savin \cite{MooneySavin}.

    Hessian and gradient estimates for the special Lagrangian equation in two dimensions with critical and supercritical phase were first proven by Warren-Yuan \cite{MY2009}. In higher dimensions, Hessian estimates have been shown for the special Lagrangian equation by Warren-Yuan \cite{WY}, Wang-Yuan \cite{WaY}, Li \cite{Lcomp}, and most recently, using the doubling method by Shankar \cite{RSsLag}. Hessian estimates for the Lagrangian mean curvature equation with critical and supercritical phase $\Theta\in C^{1,1}$ was shown by Bhattacharya \cite{AB, AB2d}. See also Lu \cite{Siyuan}. Hessian estimates were later extended by Zhou \cite{Zhou1} to a supercritical Lipschitz continuous phase. Hessian estimates for critical and supercritical Lipschitz continuous phase were shown by Ding \cite{Ding1} using geometric measure theoretic tools.

    The difficulty in proving Hessian estimates for the Lagrangian mean curvature equation in two dimensions comes from the fact that when $\Theta$ is zero, the function $u$ becomes harmonic. Harmonic functions, while known to be regular, do not satisfy a Jacobi inequality and hence, the Jacobi inequality must degenerate as $\Theta$ approaches zero. This degeneracy prevents us from applying the direct integral techniques used to prove Hessian estimates in higher dimensions $n \geq 3$.  A similar degeneration of the Jacobi inequality occurs for the $\sigma_2=1$ equation in dimension four \cite{RY3}. We would like to point out that while the Jacobi inequality given in \cite[(2.2)]{AB2d} remains valid in the supercritical case, once the phase touches the critical value, the constant $c(2)$ is no longer purely dimension-dependent. This phenomenon is specific to dimension two only.

   A second difficulty is that when $\Theta$ crosses the critical value, both eigenvalues may change sign. Previous works avoided this by assuming either $\Theta>0$ or $\Theta<0$, where the estimates applied separately. Pointwise, however, one can reduce to the symmetric opposite case since the potential $u$ also satisfies
    \[
    \arctan(-\lambda_1) + \arctan(-\lambda_2) = -\Theta(x).
    \]
    The pointwise calculations are shown invariant under this sign change, so a symmetry argument takes care of negative phase values.  The condition that $D\Theta = 0$ on the level set $\{\Theta = 0\}$ allows us to use the pointwise interpolation inequality \cite[Lemma 7.7.2]{lhorm1}. This assumption on $\Theta$ means that the function $g(x) = \max\{0,\Theta(x)\}$ (or $g(x) = \max\{0,-\Theta(x)\}$) will be a nonnegative $C^{1,1}$ function for which the interpolation inequality still holds.

The doubling method of Shankar-Yuan \cite{RY3}, used to obtain Hessian estimates for the $\sigma_2$ equation in dimension four, fails for \eqref{s} because the entries of the linearized $\sigma_2$ operator behave like $O(|D^2u|)$ near degeneracy, whereas the arctangent operator behaves like $O(\frac{1}{1+\lambda_n^2})$, so it shrinks. The good terms are only $O(1)$, while the bad terms are $O(\frac{1}{1+\lambda_n^2}/\rho^2)$, creating a delicate balance between the smallness of the operator and the largeness of $1/\rho$. Unlike in the $\sigma_2$ case, where abundant good terms compensated in the degenerate region, the setting of \eqref{s} requires a finer analysis of the degeneracy.  The cutoff in the test function requires adjusting to a higher rate of decay to the boundary.

An important ingredient in the proof is an Alexandrov Theorem that holds for convex functions (see \cite{EG}), and was extended by Chaudhuri-Trudinger \cite{CT} to functions which are $k$-convex where $k\geq n/2$, and $n$ is the dimension. We recall that a function $u$ is $k$-convex if $\sigma_j\geq 0$ for all $1\leq j \leq k$, where $\sigma_j$ is the $j$-th symmetric polynomial of the eigenvalues of $D^2u$. However, in two dimensions, the potential $u$ is only $1$-convex (subharmonic) for positive phase, which does not satisfy the above conditions. In order to prove an Alexandrov Theorem for viscosity solutions of \eqref{s}, we develop a new scale-invariant, interior gradient estimate and an integral estimate for the volume form $V = \sqrt{\det g}$. We note that this bound of the volume form does not require the assumption that the gradient of the phase vanishes on its zero set.
    
    Lastly, to extend the regularity provided by the Alexandrov Theorem, we make use of a recent result by Fan \cite{ZF} which extends Savin's \cite{Savin} perturbation theory to second order, fully nonlinear equations with variable right-hand side. As in \cite{RY3}, this shows that the regularity at a point $y_0$ can be extended to regularity over a small ball $B_{r}(y_0)$. The doubling inequality then extends this regularity to a larger ball containing the singular point, contradicting the assumed irregularity.

\subsection*{Organization}The paper is organized as follows. In Section \ref{sec_prel}, we introduce notation and recall several useful equations. In Section \ref{sec_jac}, we establish the degenerate Jacobi inequality, which is then applied in Section \ref{sec_doubling} to prove a doubling property for the Hessian. In Section \ref{sec_alexandrv}, we develop an Alexandrov-type Theorem by deriving both a scale-invariant gradient estimate and an integral estimate for the volume. In Section \ref{sec_main}, we present the proof of Theorem \ref{thm1}. The paper concludes with an Appendix containing additional computations. 

\subsection*{Acknowledgments} 
AB acknowledges
the support of NSF grant DMS-2350290, the Simons Foundation grant MPS-TSM-00002933, and a Bill Guthridge fellowship from UNC-Chapel Hill. JW acknowledges
the support of NSF RTG grant DMS-2135998 and NSF grant DMS-2350290.

\section{Preliminaries}\label{sec_prel}

For the convenience of the readers, we recall some preliminary results. We first introduce some notation that will be used in this paper. The induced Riemannian metric on the Lagrangian submanifold $L=(x,Du(x))\subset \mathbb{R}^2\times\mathbb{R}^2$ is given by
\[
g=I_2+(D^2u)^2 .
\]
We denote
 \begin{align*} 
    \partial_i=\frac{\partial}{\partial x_i} \text{ , }
     \partial_{ij}=\frac{\partial^2}{\partial x_i \partial x_j} \text{ , }
     u_i=\partial_iu \text{ , }
    u_{ij}=\partial_{ij}u.
    \end{align*}
  Note that for the functions defined below, the subscripts on the left do not represent partial derivatives\begin{align*}
    h_{ijk}=\sqrt{g^{ii}}\sqrt{g^{jj}}\sqrt{g^{kk}}u_{ijk},\quad
    g^{ii}=\frac{1}{1+\lambda_i^2}.
    \end{align*}
Here, $(g^{ij})$ is the inverse of the matrix $g$, and $h_{ijk}$ denotes the second fundamental form when the Hessian of $u$ is diagonalized.
The volume form, gradient, and inner product with respect to the metric $g$ are given by
\begin{align*}
    dv_g=\sqrt{\det g}dx &= Vdx \text{ , }\qquad
    \nabla_g v=g^{ij}v_iL_j\\
    \langle\nabla_gv,\nabla_g w\rangle_g &=g^{ij}v_iw_j \text{ , }\quad
    |\nabla_gv|^2=\langle\nabla_gv,\nabla_g v\rangle_g.
\end{align*}

The Laplace-Beltrami operator on the non-minimal submanifold $L=(x, Du(x))$ is given by:
\begin{equation*}
\Delta_g =g^{ij}\partial_{ij}-g^{jp}u_{pq}(\pd_q\Theta)  \partial_j.
\end{equation*}

On taking the gradient of both sides of the Lagrangian mean curvature equation \eqref{s}, we get another useful equation
\begin{equation}
\sum_{a,b=1}^{2}g^{ab}u_{jab}=\pd_j\Theta .\label{linearize}
\end{equation}
The details for the two above equations can be found in \cite{AB}.

\section{A degenerate Jacobi inequality}\label{sec_jac}

In this section, we prove a new Jacobi-type inequality for the slope of the gradient graph $(x,Du(x))$ in two dimensions, i.e., we show that a certain function of the slope of the gradient graph $(x,Du(x))$ is almost strongly subharmonic.

\begin{lemma}\label{PropLapl}
Let $u$ be a smooth solution of \eqref{s} in $\mathbb{R}^{2}$. Suppose that the Hessian $D^{2}u$ is diagonalized at the point $x_0$. Then at $x_0$, $b=\log\sqrt{\det g}$ satisfies
\begin{align}
    \Delta_g b &\overset{x_0}{=}\sum_{a,b,c=1}^2(1 + \lambda_b\lambda_c)h_{abc}^2\label{mcterms}\\
    &\quad +\sum_{i=1}^2 g^{ii}\lambda_i\pd_{ii}\Theta - \sum_{i=1}^2 g^{ii}\lambda_i(\pd_i\Theta) \pd_i b.\label{psiterms}
\end{align}
\end{lemma} 

The proof follows from the computations in \cite[Lemma 2.1]{YBern} and \cite[Proposition 3.1]{BW}.

Next, we have the following degenerate Jacobi inequality which was inspired by \cite{MY2009}.
\begin{proposition}
    Let $u$ be a smooth solution of \eqref{s} in $B_1(0)\subset \re^2$ where $-\pi < \Theta < \pi$ is in $C^2(B_1(0))$ such that $D\Theta = 0$ on the level set $\{\Theta = 0\}$. Suppose that the Hessian $D^2u$ is diagonalized at the point $x_0$. Then at $x_0$, $b= \log\sqrt{\det g}$ satisfies
    \begin{equation}\label{degJI}
        \Delta_g b \geq \frac{\sin|\Theta(x_0)|}{4}|\nabla_g b|^2 -C(\|D^2\Theta\|_{L^\infty(B_1(0))})
    \end{equation}
    whenever $0\leq |\Theta(x_0)| < \frac{\pi}{2}$. When $|\Theta(x_0)|\geq \frac{\pi}{2}$, we get
    \begin{equation}\label{JI2}
        \Delta_g b \geq \frac{3}{8}|\nabla_g b|^2 -C(\|D\Theta\|_{L^\infty(B_1(0))},\|D^2\Theta\|_{L^\infty(B_1(0))}).
    \end{equation}
\end{proposition}

\begin{proof}
    We prove this in four cases, based on the value of $\Theta(x_0)$.
    
    \textbf{Case 1.} Suppose that $0< \Theta(x_0)\leq \frac{\pi}{2}$. We start by considering the mean curvature terms \eqref{mcterms} from Proposition \ref{PropLapl}. We will use \eqref{linearize} in the form of
    \begin{equation}\label{Dpsi}
        h_{11i} + h_{22i} = \sqrt{g^{ii}}\Theta_{x_i}
    \end{equation}
    for $i=1,2$ in order to write everything in terms of $h_{112}$ and $h_{122}$. We see that
    \begin{align*}
        \sum_{a,b,c=1}^2(1 + \lambda_b\lambda_c)h_{abc}^2 &= h^2_{111}(1 + \lambda_1^2) + h_{112}^2(3 + (\lambda_1+\lambda_2)^2 - \lambda_2^2)\\
        &\quad + h_{222}^2(1 + \lambda_2^2) + h_{122}^2(3 + (\lambda_1+\lambda_2)^2 - \lambda_1^2)\\
        &=(\sqrt{g^{11}}\Theta_{x_1} - h_{122})^2(1 + \lambda_1^2) + h_{112}^2(3 + (\lambda_1+\lambda_2)^2 - \lambda_2^2)\\
        &\quad + (\sqrt{g^{22}}\Theta_{x_2} - h_{112})^2(1 + \lambda_2^2) + h_{122}^2(3 + (\lambda_1+\lambda_2)^2 - \lambda_1^2)\\
        &= (h_{112}^2 + h_{122}^2)(4 + (\lambda_1 + \lambda_2)^2)+ (g^{11}\Theta_{x_1}^2 - 2\sqrt{g^{11}}\Theta_{x_1}h_{122})(1 + \lambda_1^2)\\ 
        &\quad + (g^{22}\Theta_{x_2}^2 - 2\sqrt{g^{22}}\Theta_{x_2}h_{112})(1 + \lambda_2^2).
    \end{align*}
    Now, the right-most term of \eqref{psiterms} becomes
    \begin{align}
        \sum_{i=1}^2 g^{ii}\lambda_i(\pd_i\Theta) \pd_i b &= \sum_{i,j=1}^2g^{ii}g^{jj}\Theta_{x_i}\lambda_i\lambda_ju_{jji}=\sum_{i,j}^2\sqrt{g^{ii}}\Theta_{x_i}\lambda_i\lambda_jh_{jji}\nonumber\\
        &=g^{11}\Theta_{x_1}^2\lambda_1^2 -\sqrt{g^{11}}\Theta_{x_1}h_{122}\lambda_1(\lambda_1-\lambda_2)\nonumber\\
        &\quad + g^{22}\Theta_{x_2}^2\lambda_2^2 -\sqrt{g^{22}}\Theta_{x_2}h_{112}\lambda_2(\lambda_2-\lambda_1)\nonumber
    \end{align}
    where the last lines follow from \eqref{Dpsi}. Combining these together, we get
    \begin{align*}
        \Delta_g b &= (h_{112}^2 + h_{122}^2)(4 + (\lambda_1 + \lambda_2)^2)+ (g^{11}\Theta_{x_1}^2 - 2\sqrt{g^{11}}\Theta_{x_1}h_{122})(1 + \lambda_1^2)\\ 
        &\quad + (g^{22}\Theta_{x_2}^2 - 2\sqrt{g^{22}}\Theta_{x_2}h_{112})(1 + \lambda_2^2)\\
        &\quad -g^{11}\Theta_{x_1}^2\lambda_1^2 +\sqrt{g^{11}}\Theta_{x_1}h_{122}\lambda_1(\lambda_1-\lambda_2)\\
        &\quad - g^{22}\Theta_{x_2}^2\lambda_2^2 +\sqrt{g^{22}}\Theta_{x_2}h_{112}\lambda_2(\lambda_2-\lambda_1)\\
        &\quad + g^{11}\lambda_1\Theta_{x_1x_1} + g^{22}\lambda_2\Theta_{x_2x_2} \\
        &= (h_{112}^2 + h_{122}^2)(4 + (\lambda_1 + \lambda_2)^2) \\
        &\quad+\sqrt{g^{11}}\Theta_{x_1}h_{122}(-2(1 + \lambda_1^2)+ \lambda_1(\lambda_1-\lambda_2))\\
        &\quad+\sqrt{g^{22}}\Theta_{x_2}h_{112}(-2(1+\lambda_2^2)+\lambda_2(\lambda_2 - \lambda_1))\\
        &\quad+g^{11}\Theta_{x_1}^2 + g^{22}\Theta_{x_2}^2\nonumber\\
        &\quad+g^{11}\lambda_1\Theta_{x_1x_1} + g^{22}\lambda_2\Theta_{x_2x_2}.
    \end{align*}

    The gradient of $b$ is then given by 
    \begin{align*}
        |\nabla_g b|^2 &= \sum_{i=1}^2g^{ii}(\pd_i b)^2 = \sum_{i=1}^2g^{ii}\left(\frac{1}{2}\sum_{a,b=1}^2g^{ab}\pd_ig_{ab}\right)^2\\
        &= \sum_{i=1}^2g^{ii}\left(\sum_{j=1}^2g^{jj}\lambda_ju_{jji}\right)^2\\
        &=g^{11}(g^{11}\lambda_1u_{111} + g^{22} \lambda_2 u_{122})^2 + g^{22}(g^{11} \lambda_1 u_{112} + g^{22}\lambda_2 u_{222})^2\\
        &= g^{11}(g^{22}u_{122}(\lambda_2-\lambda_1)+ \lambda_1\Theta_{x_1})^2 + g^{22}(g^{11}u_{112}(\lambda_1-\lambda_2) + \lambda_2\Theta_{x_2})^2\\
        &=(h_{112}^2+h_{122}^2)(\lambda_1-\lambda_2)^2 + g^{11}\lambda_1^2\Theta_{x_1}^2 + g^{22}\lambda_2^2\Theta_{x_2}^2\\
        &\quad + 2h_{122}\lambda_1(\lambda_2-\lambda_1)\sqrt{g^{11}}\Theta_{x_1} + 2h_{112}\lambda_2(\lambda_1-\lambda_2)\sqrt{g^{22}}\Theta_{x_2}.
    \end{align*}

    Altogether, we get
    \begin{align}
        \Delta_g b - \epsilon|\nabla_g b|^2 &= (h_{112}^2 + h_{122}^2)(4 + (\lambda_1 + \lambda_2)^2 - \epsilon(\lambda_1-\lambda_2)^2)\nonumber\\
        &\quad+\sqrt{g^{11}}\Theta_{x_1}h_{122}(-2(1 + \lambda_1^2)+ \lambda_1(\lambda_1-\lambda_2)-2\epsilon\lambda_1(\lambda_2-\lambda_1))\nonumber\\
        &\quad+\sqrt{g^{22}}\Theta_{x_2}h_{112}(-2(1+\lambda_2^2)+\lambda_2(\lambda_2 - \lambda_1)-2\epsilon\lambda_2(\lambda_1-\lambda_2))\nonumber\\
        &\quad+g^{11}\Theta_{x_1}^2(1-\epsilon\lambda_1^2) + g^{22}\Theta_{x_2}^2(1-\epsilon\lambda_2^2)\nonumber\\
        &\quad+g^{11}\lambda_1\Theta_{x_1x_1} + g^{22}\lambda_2\Theta_{x_2x_2}\nonumber\\
        &= (h_{112}^2 + h_{122}^2)(4 + (\lambda_1 + \lambda_2)^2 - \epsilon(\lambda_1-\lambda_2)^2)\label{hterms}\\
        &\quad+\sqrt{g^{11}}\Theta_{x_1}h_{122}(-(1-2\epsilon)\lambda_1^2-(1 + 2\epsilon)\lambda_1\lambda_2 - 2)\label{psix1}\\
        &\quad+\sqrt{g^{22}}\Theta_{x_2}h_{112}(-(1-2\epsilon)\lambda_2^2-(1 + 2\epsilon)\lambda_1\lambda_2 - 2)\label{psix2}\\
        &\quad+g^{11}\Theta_{x_1}^2(1-\epsilon\lambda_1^2) + g^{22}\Theta_{x_2}^2(1-\epsilon\lambda_2^2)\label{psixs}\\
        &\quad+g^{11}\lambda_1\Theta_{x_1x_1} + g^{22}\lambda_2\Theta_{x_2x_2}.\nonumber
    \end{align}

    From here we apply Young's inequality to \eqref{psix1} and \eqref{psix2} so that they combine into \eqref{hterms} and \eqref{psixs}. 

    Using the alternate form of equation \eqref{tanpde} given by
    \[
    0=\cos\Theta \Delta u +\sin\Theta(\det D^2u - 1) = \cos\Theta \sigma_1 + \sin\Theta(\sigma_2-1),
    \]
    we rewrite the coefficient of $(h_{112}^2 + h_{122}^2)$ as a quadratic polynomial in $\sigma_1$. We note that $\sigma_1=\lambda_1+\lambda_2\geq 0$ and $\sigma_2=\lambda_1\lambda_2\leq 0$ from the assumption that $0<\Theta(x_0)\leq\frac{\pi}{2}$. 

    Applying Young's inequality to \eqref{psix1}, we get
    \begin{align*}
        \sqrt{g^{11}}\Theta_{x_1}h_{122}&(-\lambda_1^2+2\epsilon\lambda_1^2-\lambda_1\lambda_2 - 2\epsilon\lambda_1\lambda_2 - 2) \geq\\
        &\quad-\frac{g^{11}\Theta_{x_1}^2\lambda_1^2}{2\eta}-\frac{h_{122}^2\lambda_1^2\eta}{2}
        -\epsilon g^{11}\Theta_{x_1}^2\lambda_1^2 - \epsilon h_{122}^2\lambda_1^2\\
        &\quad-\frac{g^{11}\Theta_{x_1}^2\lambda_1^2}{2\eta} - \frac{h_{122}^2\lambda_2^2\eta}{2} -\epsilon g^{11}\Theta_{x_1}^2\lambda_1^2 - \epsilon h_{122}^2\lambda_2^2 - g^{11}\Theta_{x_1}^2 - h_{122}^2.
    \end{align*}
    Applying Young's inequality again, \eqref{psix2} combines with \eqref{hterms} and \eqref{psixs} to get
    \begin{align*}
        \Delta_g b - \epsilon|\nabla_g b|^2 &\geq (h_{112}^2 + h_{122}^2)\bigg((1 - 2\epsilon - \eta/2)\sigma_1^2 -\cot\Theta(\eta + 6\epsilon)\sigma_1 + (3 + \eta + 6\epsilon)\bigg)\\
        &\quad -g^{11}\Theta_{x_1}^2(1 + \lambda_1^2(2\epsilon + 1/\eta)) - g^{22}\Theta_{x_1}^2(1 + \lambda_2^2(2\epsilon + 1/\eta))\\
        &\quad +g^{11}\lambda_1\Theta_{x_1x_1} + g^{22}\lambda_2\Theta_{x_2x_2}.
    \end{align*}

    For the coefficient on the mean curvature terms to be positive it suffices to pick $\epsilon,\eta>0$ such that
    \begin{align}
        (1-2\epsilon-\eta/2)&\geq 0\nonumber\\
        \cot^2\Theta(\eta + 6\epsilon)^2 &\leq 4(1-2\epsilon-\eta/2)(3 + \eta + 6\epsilon)\label{cotThet}
    \end{align}
    where the second line comes from the quadratic formula of $a\sigma_1^2 + b\sigma_1 + c$ and forcing $b^2 - 4ac \leq 0$ so that $0\leq a\sigma_1^2 + b\sigma_1 + c$ for any $\sigma_1\geq 0$. 

    We choose $\epsilon=(\sin\Theta)/4,\;\eta = \sin\Theta$ and  observe:
    \begin{align*}
        4(1-\sin\Theta)(3+\frac{5}{2}\sin\Theta)-(5/4)^2\cos^2\Theta &= 4(3 - \frac{1}{2}\sin \Theta - \frac{5}{2}\sin^2\Theta)- \frac{25}{4}\cos^2\Theta\\
        &= 12 - 2\sin\Theta - 10\sin^2\Theta - \frac{25}{4}\cos^2\Theta\\
        &= 12 - \frac{25}{4} - 2\sin\Theta - \frac{15}{4}\sin^2\Theta\\
        &= \frac{1}{4}(23 - 8\sin\Theta - 15\sin^2\Theta)\geq 0
    \end{align*}
    which is equivalent to \eqref{cotThet}.
    
    We get
    \begin{equation}\label{D2-sD1}
        \Delta_g b - \frac{\sin\Theta}{4}|\nabla_g b|^2\geq -|D\Theta|^2(3/2 + \csc\Theta) + g^{11}\lambda_1\Theta_{x_1x_1} + g^{22}\lambda_2\Theta_{x_2x_2}.
    \end{equation}

    As $D\Theta = 0$ on $\{\Theta =0\}$, we can apply the pointwise interpolation inequality \cite[Lemma 7.7.2]{lhorm1} to $\max\{0,\Theta\}$, which is still $C^{1,1}(B_1(0))$, so \eqref{D2-sD1} becomes
    \begin{align*}
        \Delta_g b - \frac{\sin\Theta}{4}|\nabla_g b|^2&\geq -|D\Theta|^2(3/2 + \csc\Theta) + g^{11}\lambda_1\Theta_{x_1x_1} + g^{22}\lambda_2\Theta_{x_2x_2}\\
        &\geq -(\Theta(x_0) + 2\|D^2\Theta\|_{L^\infty(B_1(0))})\left(\frac{3\Theta(x_0)}{2} + \frac{\Theta(x_0)}{\sin(\Theta(x_0))}\right) - 2\|D^2\Theta\|_{L^\infty(B_1(0))}\\
        &\geq -\left(\frac{\pi}{2} + 2\|D^2\Theta\|_{L^\infty(B_1(0))}\right)\left(\frac{3\pi}{4} + \frac{\pi}{2} \right) - 2\|D^2\Theta\|_{L^\infty(B_1(0))}\\
        &\geq -\frac{5\pi^2}{8} - \left(\frac{5\pi + 4}{2}\right)\|D^2\Theta\|_{L^\infty(B_1(0))}
    \end{align*}
    where we use $\Theta(x_0)\leq \frac{\pi}{2}$.

    \textbf{Case 2.} Suppose that $\frac{\pi}{2}< \Theta(x_0)<\pi$. As $\Theta$ is in the hypercritical range, we get that $\sigma_1,\sigma_2\geq 0$. Returning to the mean curvature terms \eqref{mcterms} from Proposition \ref{PropLapl}, we get
    \begin{align*}
        \sum_{a,b,c=1}^2(1 + \lambda_b\lambda_c)h_{abc}^2 &= h^2_{111}(1 + \lambda_1^2) + h_{112}^2(3 + \lambda_1^2 + 2\lambda_1\lambda_2)\\
        &\quad + h_{222}^2(1 + \lambda_2^2) + h_{122}^2(3 + \lambda_2^2 + 2\lambda_1\lambda_2).
    \end{align*}
    We also see that the gradient of $b$ is 
    \begin{align*}
        \frac{1}{2}|\nabla_g b|^2 &= \sum_{i=1}^2g^{ii}(\pd_i b)^2 = \sum_{i=1}^2g^{ii}\left(\frac{1}{2}\sum_{a,b=1}^2g^{ab}\pd_ig_{ab}\right)^2\\
        &= \sum_{i=1}^2g^{ii}\left(\sum_{j=1}^2g^{jj}\lambda_ju_{jji}\right)^2\\
        &=g^{11}(g^{11}\lambda_1u_{111} + g^{22} \lambda_2 u_{122})^2 + g^{22}(g^{11} \lambda_1 u_{112} + g^{22}\lambda_2 u_{222})^2\\
        &\leq h_{111}^2\lambda_1^2 + h_{112}^2\lambda_1^2 + h_{122}^2\lambda_2^2 + h_{222}^2\lambda_2^2.
    \end{align*}
    Thus,
    \begin{align*}
        \Delta_g b - \frac{1}{2}|\nabla_g b|^2 &\geq h_{111}^2 + h_{112}^2(3 + 2\lambda_1\lambda_2) + h_{122}^2(3 + 2\lambda_1\lambda_2) + h_{222}^2\\
        &\qquad +\sum_{i=1}^2 g^{ii}\lambda_i\pd_{ii}\Theta - \sum_{j=1}^2 g^{jj}\lambda_j(\pd_j\Theta)\pd_j b\\
        &\geq -\frac{\eta}{2}|\nabla_g b|^2 - \frac{1}{2\eta}\sum_{j=1}^{2}g^{jj}\lambda_j^2\Theta_{x_j}^2+\sum_{i=1}^2 g^{ii}\lambda_i\pd_{ii}\Theta\\
        &\geq -\frac{\eta}{2}|\nabla_g b|^2 - \frac{1}{2\eta}|D\Theta(x_0)|^2 - 2|D^2\Theta(x_0)|.
    \end{align*}
    Taking $\eta = 1/4$ we get
    \begin{align*}
        \Delta_g b &\geq \frac{3}{8}|\nabla_g b|^2  - 2|D\Theta(x_0)|^2- 2|D^2\Theta(x_0)|\\
        &\geq \frac{3}{8}|\nabla_g b|^2 - 2\|D\Theta\|^2_{L^\infty(B_1(0))} - 2\|D^2\Theta\|_{L^\infty(B_1(0))}.
    \end{align*}

    \textbf{Case 3.} Suppose that $-\pi<\Theta(x_0)< 0$. As the function $b$ is symmetric in the eigenvalues, we can reduce to the positive $\Theta$ cases above since $u$ also solves
    \[
    \arctan(-\lambda_1) + \arctan(-\lambda_2) = -\Theta(x)
    \]
    where now $-\Theta(x_0)> 0$. All of the estimates above hold with the following replacements:
    \begin{itemize}
        \item $\Theta(x_0)$ becomes $-\Theta(x_0)$
        \item $\lambda_1$ becomes $-\lambda_2$
        \item $\lambda_2$ becomes $-\lambda_1$
        \item $\max\{0,\Theta\}$ becomes $\max\{0, -\Theta\}$.
    \end{itemize}
    
    \textbf{Case 4.} Suppose that $\Theta(x_0) = 0$. From the assumption on $\Theta$, we get $D\Theta(x_0) = 0$, and \eqref{Dpsi} becomes
    \[
    h_{11i}(x_0) + h_{22i}(x_0) = \pd_i\Theta(x_0) = 0.
    \]
    Thus, the Jacobi inequality follows from \cite{MY2009} as
    \begin{align*}
        \Delta_g b &\geq \sin(\Theta(x_0))|\nabla_g b|^2 +g^{11}\lambda_1\Theta_{x_1x_1} + g^{22}\lambda_2\Theta_{x_2x_2}\\
        &\geq -2\|D^2\Theta\|_{L^\infty(B_1(0))}
    \end{align*}
    which satisfies \eqref{degJI} with $\Theta(x_0) = 0$.
\end{proof}

\section{Doubling property of the Hessian}\label{sec_doubling}
Using the above degenerate Jacobi inequality, we prove the following doubling inequality.

\begin{lemma}\label{doubling}
    Let $u$ be a smooth solution of \eqref{s} on $B_2(0)\subset\re^2$ with $-\pi< \Theta < \pi$ and $D\Theta = 0$ on the level set $\{\Theta = 0\}$. For any $p\in B_{r}(0)$, with $r \leq 1$,
    \[
    \sup_{B_{r}(p)}b(D^2u)\leq C(r,n,\|u\|_{C^{1}(B_2(0))}, \|D\Theta\|_{L^\infty(B_2(0))},\|D^2\Theta\|_{L^\infty(B_2(0))})\sup_{B_\frac{r}{2}(p)}b(D^2u),
    \]
    and hence, by the properness of $b$, 
    \[
    \sup_{B_{r}(p)}|D^2u|\leq C(r,n,\|u\|_{C^{1}(B_2(0))},\|D\Theta\|_{L^\infty(B_2(0))}, \|D^2\Theta\|_{L^\infty(B_2(0))},\sup_{B_\frac{r}{2}(p)}|D^2u|).
    \]
\end{lemma}
\begin{proof}
  We take $r=1$ and $p=0$, since the rest of the proof follows for general $r$ and $p$. Define
    \begin{equation}\label{testP}
        P := \nu\ln\rho(x) + \alpha ( x\cdot Du - u) + \beta|Du|^2/2 +\ln\max(\bar{b},\gamma^{-1})
    \end{equation}
    where $\rho(x) = 1 - |x|^2$, $ \bar{b} = b - \max_{B_{1/2}(0)}b$, and constants $0 < |\alpha| < 1$, $0 < \beta < 1$, $\nu> 1$, $0< \gamma < 1$ to be chosen later. We will let $\Gamma = 1 + \|u\|_{C^{1}(B_2(0))}$. We will assume that $|x|\geq 1/2$ (since $\bar{b}$ takes care of $|x|\leq 1/2$) and $\bar{b}>\gamma^{-1}$ (since otherwise we are done).

    Assume $P$ achieves its maximum at some point $y \in B_1(0)$. We will assume that $D^2u$ is diagonalized at $y$. The first and second derivatives of $P$ are given by
    \begin{align*}
        P_i &= \frac{-2\nu x_i}{\rho }+ \alpha x_k u_{ki} + \beta u_ku_{ki} + \frac{b_i}{\bar{b}} \\
        P_{ij}&= -2\nu\frac{\delta_{ij}}{\rho} - 4\nu\frac{x_ix_j}{\rho^2} + \alpha u_{ij} + \alpha x_k u_{ijk} + \beta u_{kj}u_{ki} + \beta u_{k}u_{ijk} + \frac{b_{ij}}{\bar{b}} - \frac{b_i b_j}{\bar{b}^2}.
    \end{align*}
    At the maximum point $y$, the above reduces to
    \begin{align*}
        0 &= P_i = \frac{-2\nu y_i}{\rho} + \alpha y_i\lambda_i + \beta u_i\lambda_i + \frac{b_i}{\bar{b}}\\
        0 &\geq g^{ij}P_{ij} = \frac{-2\nu g^{ii}}{\rho} - 4\nu \frac{g^{ii}y_i^2}{\rho^2} + \alpha g^{ii}\lambda_i + \alpha y\cdot D\Theta \nonumber\\
        &\qquad + \beta g^{ii}\lambda_i^2 + \beta Du\cdot D\Theta +  g^{ij}\frac{b_{ij}}{\bar{b}} -  g^{ij}\frac{b_ib_j}{\bar{b}^2} \nonumber\\
        &\geq \frac{-2\nu g^{ii}}{\rho} - 4\nu \frac{g^{ii}y_i^2}{\rho^2} + \alpha g^{ii}\lambda_i + \alpha y\cdot D\Theta \nonumber\\
        &\qquad + \beta g^{ii}\lambda_i^2 + \beta Du\cdot D\Theta + \bar{b}g^{ii}\left(\frac{b_i}{\bar{b}}\right)^2\left(\epsilon - \frac{1}{\bar{b}}\right) +  g^{ii}\lambda_i \Theta_i \left(\frac{b_i}{\bar{b}}\right) - \frac{C(\Theta)}{\bar{b}}
    \end{align*}
    where the last line comes from the Jacobi inequality \eqref{degJI} or \eqref{JI2} with $\epsilon = \sin(|\Theta(y)|)/4$ or $\epsilon = 3/8$ and their respective constants $C(\Theta)$. Rearranging the above inequalities, we get
    \begin{align}
        \frac{b_i}{\bar{b}} &= 2\nu\frac{y_i}{\rho} - \alpha y_i\lambda_i - \beta u_i\lambda_i \label{gradP2}\\
        2\nu\frac{g^{ii}}{\rho} + 4\nu\frac{g^{ii}y_i^2}{\rho^2}&\geq \alpha g^{ii}\lambda_i + \beta g^{ii}\lambda_i^2 + (\alpha y +  \beta Du)\cdot D\Theta \nonumber\\
        &\qquad + \bar{b}g^{ii}\left(\frac{b_i}{\bar{b}}\right)^2\left(\epsilon - \frac{1}{\bar{b}}\right) +  g^{ii}\lambda_i\Theta_i \left(\frac{b_i}{\bar{b}}\right) - \frac{C(\Theta)}{\bar{b}}. \label{HessP2}
    \end{align}

    \textbf{Case 1.} First, suppose that $0\leq \Theta(y) < \frac{\pi}{2}$. It follows that $\lambda_1\geq|\lambda_2|$ and $\lambda_2 < 0$. We can further assume that $P>>1$ (as otherwise $P < C$ and we are done), and hence $\lambda_1>>1$.
    \begin{enumerate}
        \item[Case 1.1.] We first assume $\frac{\sin \Theta}{4} = \epsilon \leq \tau/\bar{b}$ with $\tau = o(\bar{b}) = \bar{b}^q$ for some $0 < q < 1$ to be chosen later and that $\epsilon$ is small, possibly zero. Hence $|\lambda_2|>>1$ as well. 
        
        Plugging \eqref{gradP2} into \eqref{HessP2}, we get
        \begin{align}
            \bar{b}\left(\epsilon - \frac{1}{\bar{b}}\right)g^{ii}\left(\frac{b_i}{\bar{b}}\right)^2 &\geq -g^{ii}\left(\frac{b_i}{\bar{b}}\right)^2\nonumber\\
            &\geq -3g^{ii}\left(4\nu^2\frac{y_i^2}{\rho^2}+ \alpha^2 y_i^2\lambda_i^2 + \beta^2u_i^2\lambda_i^2\right)\nonumber\\
            &\geq -12g^{ii}\nu^2\frac{y_i^2}{\rho^2} - 3\alpha^2 - 3\beta^2\Gamma^2\label{bisquare}
        \end{align}
        and
        \begin{align}
             g^{ii}\lambda_i\Theta_i\frac{b_i}{\bar{b}} &= g^{ii}\lambda_i\Theta_i\frac{2\nu y_i}{\rho} - g^{ii}\lambda_i^2\Theta_i\alpha y_i - g^{ii}\lambda_i^2\Theta_i\beta u_i\nonumber\\
            &\geq -2\nu^2 \frac{g^{ii}y_i^2}{\rho^2} - \frac{1}{2}(\alpha^2 + \beta^2\Gamma^2) - |D\Theta(y)|.^2\label{theta1}
        \end{align}
        Lastly, we bound
        \begin{equation}\label{thetadot}
            (\alpha y +  \beta Du)\cdot D\Theta \geq - \frac{1}{2}(\alpha^2 + \beta^2\Gamma^2) - \frac{1}{2}|D\Theta(y)|^2.
        \end{equation}
        Hence, using \eqref{bisquare}, \eqref{theta1}, and \eqref{thetadot}, the inequality in \eqref{HessP2} becomes
        \begin{align}
            2\nu\frac{g^{ii}}{\rho} &+ (4\nu+14\nu^2)\frac{g^{ii}y_i^2}{\rho^2}-\alpha g^{ii}\lambda_i + \frac{3}{2}|D\Theta(y)|^2 + \frac{C(\|D^2\Theta\|_{L^\infty(B_1(0))})}{\bar{b}} \nonumber\\
            &\qquad\geq  \beta/2 - 4(\alpha^2 +  \beta^2\Gamma^2).\label{rhs}
        \end{align}

        We now seek to bound
        \begin{equation}\label{lhs}
            2\nu\frac{g^{ii}}{\rho} + (4\nu+14\nu^2)\frac{g^{ii}y_i^2}{\rho^2}-\alpha g^{ii}\lambda_i + \frac{3}{2}|D\Theta(y)|^2 + \frac{C(\|D^2\Theta\|_{L^\infty(B_1(0))})}{\bar{b}} 
        \end{equation}
        from above. From \eqref{s}, we also get that
        \begin{equation}\label{s2}
            \Theta = \arctan(-1/\lambda_1) + \arctan(-1/\lambda_2).
        \end{equation}
        From \eqref{s2}, taking $\epsilon=\frac{\sin\Theta}{4}$ to be small we get
        \begin{equation}\label{thetab}
            \frac{\Theta }{8}\leq \frac{\sin\Theta }{4}\leq \epsilon \leq \frac{\tau}{\bar{b}} = \frac{1}{\bar{b}^{1-q}},
        \end{equation}
        from which it follows that
        \begin{equation*}
            \arctan(1/|\lambda_2|)\leq 8\frac{\tau}{\bar{b}} + \arctan(1/\lambda_1).
        \end{equation*}
        Using the fact that $|\lambda_2|>1$, we get
        \begin{equation}\label{lambda2}
            \arctan(1/|\lambda_2|)\geq \arctan(1)\frac{1}{|\lambda_2|} = \frac{\pi}{4}\frac{1}{|\lambda_2|}.
        \end{equation}
        Similarly, from $\lambda_1 > 1$, it follows that
        \begin{equation}\label{lambda1}
            \arctan(1/\lambda_1)\leq C(p)\frac{1}{\bar{b}^p}
        \end{equation}
        for any $0 < p \leq 1$. The full calculations of \eqref{lambda2} and \eqref{lambda1} are shown in the Appendix. Altogether this yields
        \begin{equation}\label{lamb2}
            \frac{1}{|\lambda_2|}\leq C(q)\frac{1}{\bar{b}^{1-q}}.
        \end{equation}
        The two terms with $\rho$ in \eqref{lhs} are bounded above by 
        \begin{equation}\label{lhsg}
            2\nu\frac{g^{ii}}{\rho} + (4\nu+14\nu^2)\nu\frac{g^{ii}y_i^2}{\rho^2}\leq 20\nu^2\frac{g^{22}}{\rho^2}\leq 20\nu^2\frac{1}{\rho^2|\lambda_2|^2}\leq 20\nu^2\frac{1}{\rho^2|\lambda_2|}.
        \end{equation}
        The $\alpha$ term in \eqref{lhs} is bounded above by
        \begin{equation}\label{lhsa}
            -\alpha g^{11}\lambda_1 - \alpha g^{22}\lambda_2 \leq 2|\alpha| \frac{|\lambda_2|}{1 + \lambda_2^2}\leq |\alpha| \frac{2}{|\lambda_2|}\leq |\alpha|\frac{2}{\rho^2|\lambda_2|}.
        \end{equation}
        We then bound the fourth term in \eqref{lhs} above by
        \begin{equation}\label{lhsthet1}
            \frac{3}{2}|D\Theta(y)|^2\leq \frac{3}{2}\Theta(y)\left(\Theta(y) + 2\|D^2\Theta\|_{L^\infty(B_1(0))}\right)\leq \frac{C(\|D^2\Theta\|_{L^\infty(B_1(0))})}{\rho^2\bar{b}^{1-q}}
        \end{equation}
        where we again use \eqref{thetab} and apply the interpolation inequality from \cite[Lemma 7.7.2]{lhorm1} to the function $\max\{0, \Theta\}$ which is $C^{1,1}$ by the assumption $D\Theta = 0$ on $\{\Theta = 0\}$.
        
        We can bound the remaining term of \eqref{lhs} by
        \begin{equation}\label{lhsthet2}
            \frac{C(\|D^2\Theta\|_{L^\infty(B_1(0))})}{\bar{b}}\leq \frac{C(\|D^2\Theta\|_{L^\infty(B_1(0))})}{\rho^2\bar{b}^{1-q}}
        \end{equation}
        since $\rho < 1$ and $\bar{b}>1$.
        
        By combining \eqref{lhsg}, \eqref{lhsa},\eqref{lhsthet1},\eqref{lhsthet2}, and \eqref{lamb2} with \eqref{lhs} and \eqref{rhs}, we get
        \begin{equation}\label{beta}
             \beta/2 - 4(\alpha^2 +  \beta^2\Gamma^2) \leq \frac{C(\nu,q,|\alpha|,\|D^2\Theta\|_{L^\infty(B_1(0))})}{\rho^2 \bar{b}^{1-q}}.
        \end{equation}
        Now choosing
        \begin{equation}\label{albet}
            \alpha^2 < \frac{\beta}{292}\gamma^{2/3} < \frac{\beta}{32} \quad\text{and}\quad \beta\Gamma^2 < \frac{1}{292}\gamma^{2/3}< \frac{1}{32},
        \end{equation}
        equation \eqref{beta} yields
        \begin{equation}
            \rho \leq \frac{C(\nu,q,|\alpha|,\|D^2\Theta\|_{L^\infty(B_1(0))})}{\sqrt{\beta}\bar{b}^\frac{1-q}{2}}.
        \end{equation}
        Hence,
        \begin{align*}
            e^P\leq C(|\alpha|,\beta,\Gamma)\rho^\nu \bar{b}&\leq C(\nu,|\alpha|,\beta,\Gamma,\|D^2\Theta\|_{L^\infty(B_1(0))}) \frac{\bar{b}}{\bar{b}^\frac{\nu(1-q)}{2}}\\
            &\leq C(\nu,|\alpha|,\beta,\Gamma,\|D^2\Theta\|_{L^\infty(B_1(0))})
        \end{align*}
        for $1 - \nu(1-q)/2 = 0$. We will take $\nu = 6, q = 2/3$ from now on to get that $e^P$, and hence $P$, is bounded.

    \item[Case 1.2.] Now, assume that $\epsilon\bar{b}\geq \tau = \bar{b}^\frac{2}{3}$ and at the maximum point that $|
    y_1|\geq |y|/\sqrt{2}\geq 1/(2\sqrt{2})$. We note that now $\epsilon > 0$. We may assume $\bar{b} >\frac{2}{\epsilon}$ so that the inequality in \eqref{HessP2} becomes
        \begin{align*}
            \frac{72}{\rho^2} + |\alpha| &\geq (\alpha y + \beta Du)\cdot D\Theta + \frac{\epsilon}{2}\bar{b}g^{ii}\left(\frac{b_i}{\bar{b}}\right)^2 + g^{ii}\lambda_i\Theta_i\left(\frac{b_i}{\bar{b}}\right) - \frac{C(\|D^2\Theta\|_{L^\infty(B_1(0))})}{\bar{b}}\\
            &\geq -\frac{1}{2}(\alpha^2 + \beta^2\Gamma^2) -3|D\Theta(y)|^2\\
            &\qquad\qquad+ \frac{\bar{b}^\frac{2}{3}}{4}g^{ii}\left(\frac{b_i}{\bar{b}}\right)^2  - \frac{C(\|D^2\Theta\|_{L^\infty(B_1(0))})}{\bar{b}}
        \end{align*}
        where the first line uses
        \[
        12\frac{g^{ii}}{\rho}+24\frac{g^{ii}y_i^2}{\rho^2} < g^{22}\frac{72}{\rho^2} < \frac{72}{\rho^2}
        \]
        and
        \[
        -\alpha g^{11}\lambda_1 - \alpha g^{22}\lambda_2 \leq 2|\alpha| \frac{|\lambda_2|}{1 + \lambda_2^2}\leq |\alpha|.
        \]
        The last line uses Young's inequality
        \[
        g^{ii}\lambda_i\Theta_i\left(\frac{b_i}{\bar{b}}\right) \geq -\frac{\bar{b}^\frac{2}{3}}{4}g^{ii}\left(\frac{b_i}{\bar{b}}\right)^2 -\frac{1}{\bar{b}^\frac{2}{3}}|D\Theta(y)|^2 \geq -\frac{\bar{b}^\frac{2}{3}}{4}g^{ii}\left(\frac{b_i}{\bar{b}}\right)^2 -|D\Theta(y)|^2
        \]
        and the fact that $1 < \bar{b}^\frac{2}{3}$. We rearrange the above inequality to get
        \begin{align}
            \frac{72}{\rho^2} + |\alpha|&+ \frac{C(\|D^2\Theta\|_{L^\infty(B_1(0))})}{\rho^2\bar{b}} + \frac{1}{2}(\alpha^2 + \beta^2\Gamma^2) + |D\Theta(y)|^2\nonumber\\
            & \geq  \frac{\bar{b}^\frac{2}{3}}{4}g^{ii}\left(\frac{b_i}{\bar{b}}\right)^2\geq  \frac{\bar{b}^\frac{2}{3}}{4}g^{11}\left(\frac{b_1}{\bar{b}}\right)^2.\label{cs2rho}
        \end{align}
        Using \eqref{gradP2}, we get 
        \begin{equation}\label{case2}
            \left(\frac{b_1}{\bar{b}}\right)^2 = \left(y_1\left(\frac{12}{\rho} - \alpha\lambda_1\right) - \beta u_1\lambda_1\right)^2.
        \end{equation}
        Here, we can assume that $\frac{|\alpha|}{24}\lambda_1 > 1/\rho$ since otherwise $e^P=\rho^6\bar{b}\leq \rho \lambda_1 < \frac{24}{|\alpha|}$ and $P\leq C(|\alpha|)$. Thus, \eqref{case2} becomes
        \begin{align}
            \left(\frac{b_1}{\bar{b}}\right)^2 &= \left(y_1\left(\frac{12}{\rho} - \alpha\lambda_1\right) - \beta u_1\lambda_1\right)^2\nonumber\\
            &=  y_1^2\left(\frac{12}{\rho} - \alpha\lambda_1\right)^2 - 2y_1\left(\frac{12}{\rho} - \alpha\lambda_1\right)\beta u_1\lambda_1 - \beta^2 u_1^2\lambda_1^2\nonumber\\
            &\geq y_1^2\alpha^2\lambda_1^2/8 -2\beta^2\Gamma^2\lambda_1^2 \nonumber\\
            &\geq (\alpha^2/64 - 2\beta^2\Gamma^2)\lambda_1^2\label{c2alph1}\\
            &\geq \alpha^2\lambda_1^2/128\label{c2alph2}
        \end{align}
        where \eqref{c2alph1} follows from the assumption on $|y_1|$ and \eqref{c2alph2} follows from choosing
        \begin{equation}\label{albet2}
            \beta\Gamma < \frac{|\alpha|}{16}
        \end{equation}
        which is consistent with \eqref{albet}. We now bound \eqref{cs2rho} below by
        \begin{align*}
            C(|\alpha|,\beta,\Gamma,\|D^2\Theta\|_{L^\infty(B_1(0))})&\geq \rho^2\bar{b}^\frac{2}{3}g^{11}\left(\frac{b_1}{\bar{b}}\right)^2\\
            &\geq \rho^2\bar{b}^\frac{2}{3}g^{11}\alpha^2\lambda_1^2/128\\
            &\geq \rho^2\bar{b}^\frac{2}{3}\alpha^2/256
        \end{align*}
        where the last line follows $g^{11}\lambda_1^2>1/2$ as we assume $\lambda_1>1$. Hence,
        \begin{align*}
            e^P &\leq C(|\alpha|,\beta,\Gamma)\rho^6\bar{b} \leq C(|\alpha|,\beta,\Gamma)\rho^3\bar{b}\\
            &= C(|\alpha|,\beta,\Gamma)(\rho^2\bar{b}^\frac{2}{3})^\frac{3}{2} \leq C(|\alpha|,\beta,\Gamma,\|D^2\Theta\|_{L^\infty(B_2(0))}),
        \end{align*}
        so $P$ is again bounded.
        
    \item[Case 1.3.] Assume again that $\epsilon\bar{b} \geq \tau = \bar{b}^\frac{2}{3}$ but now at the maximum point, $|y_2|\geq |y|/\sqrt{2}\geq 1/(2\sqrt{2})$. We can again assume that $\bar{b} > \frac{2}{\epsilon}$. We may also assume that $\rho|\lambda_2| \leq \frac{24}{|\alpha|}$ as otherwise we follow the method in case 1.2 above. The inequality in \eqref{HessP2}  becomes
        \begin{align}
            g^{22}\frac{72}{\rho^2} &+ \frac{C(\|D^2\Theta\|_{L^\infty(B_1(0))})}{\bar{b}}\label{case3ineq}\\
            &\geq \frac{\beta}{2} -|\alpha|g^{11}\lambda_1 - |\alpha| g^{22}|\lambda_2| -(|\alpha| + \beta\Gamma)|D\Theta(y)| + \frac{\bar{b}^{\frac{2}{3}}}{2}g^{ii}\left(\frac{b_i}{\bar{b}}\right)^2 + g^{ii}\lambda_i\Theta_i\left(\frac{b_i}{\bar{b}}\right)^2.\nonumber
        \end{align}
    We first show that
    \begin{equation}\label{dthet}
        |D\Theta(y)|^2\leq \frac{C(\|D^2\Theta\|_{L^\infty(B_1(0))})}{\sqrt{1+\lambda_2^2}}.
    \end{equation}
    If $|\lambda_2| > 1$, then by the interpolation inequality \cite[Lemma 7.7.2]{lhorm1}, \eqref{s2}, and the calculations in the Appendix, we get
    \begin{align*}
        |D\Theta(y)|^2&\leq \Theta(y)C(\|D^2\Theta\|_{L^\infty(B_1(0))})\\
        &\leq 2\arctan\frac{1}{|\lambda_2|}C(\|D^2\Theta\|_{L^\infty(B_1(0))})\\
        &\leq \frac{C_1(\|D^2\Theta\|_{L^\infty(B_1(0))})}{\sqrt{1+\lambda_2^2}}.
    \end{align*}
    On the other hand if $|\lambda_2| < 1$, by the interpolation inequality \cite[Lemma 7.7.2]{lhorm1}, we get
    \begin{equation*}
        |D\Theta(y)|^2\leq \frac{\sqrt{2}C(\|D^2\Theta\|_{L^\infty(B_1(0))})}{\sqrt{2}}\leq \frac{C_2(\|D^2\Theta\|_{L^\infty(B_1(0))})}{\sqrt{1+\lambda_2^2}}.
    \end{equation*}
    Taking $C = \max\{C_1,C_2\}$, we get \eqref{dthet}. 
        
    Using Young's inequality with $p = 4/3$ and $q = 4$, we get
    \begin{align}
        -(|\alpha| + \beta\Gamma)|D\Theta(y)| &\geq -\frac{3}{4}(\alpha^\frac{4}{3} + \beta^\frac{4}{3}\Gamma^\frac{4}{3}) - \frac{1}{4}|D\Theta(y)|^4\nonumber\\
        &\geq -\frac{3}{4}(\alpha^\frac{4}{3} + \beta^\frac{4}{3}\Gamma^\frac{4}{3}) - g^{22}\frac{C(\|D^2\Theta\|_{L^\infty(B_1(0))})}{4}\label{youngs1}.
    \end{align}
    Using Young's inequality with $p=q = 2$, we get
    \begin{equation}\label{youngs2}
        g^{ii}\lambda_i\Theta_i\left(\frac{b_i}{\bar{b}}\right)\geq -\frac{\bar{b}^\frac{2}{3}}{4}g^{ii}\left(\frac{b_i}{\bar{b}}\right)^2 - \frac{1}{\bar{b}^\frac{2}{3}}|D\Theta(y)|^2
    \end{equation}
    and 
    \begin{align}
        -|\alpha|g^{11}\lambda_1 &\geq -\frac{1}{2\eta}\alpha^2g^{11}\lambda_1^2 - \frac{\eta}{2}g^{11}\nonumber \\
        &\geq -\frac{1}{2\eta}\frac{\beta}{292}\gamma^{2/3} - \frac{\eta}{2}g^{22}\nonumber\\
        &= -\frac{\beta}{4\cdot 292} - \gamma^{2/3}g^{22}\nonumber\\
        &\geq -\frac{\beta}{4\cdot 292} - g^{22} \label{youngs3}
    \end{align}
    where we used \eqref{albet} in the second inequality.
    
    Combining \eqref{youngs1}, \eqref{youngs2}, and \eqref{youngs3} into \eqref{case3ineq} yields
    \begin{align}
        g^{22}\frac{72}{\rho^2}& + g^{22}\frac{C(\|D^2\Theta\|_{L^\infty(B_1(0))})}{4} +  \frac{C(\|D^2\Theta\|_{L^\infty(B_1(0))})}{\bar{b}^\frac{2}{3}}\nonumber\\
        &\geq \frac{\beta}{2}\cdot\frac{291}{292}- g^{22}\frac{24}{\rho} - g^{22} -\frac{3}{4}(\alpha^\frac{4}{3} + \beta^\frac{4}{3}\Gamma^\frac{4}{3})  + \frac{\gamma^{-2/3}}{2}g^{ii}\left(\frac{b_i}{\bar{b}}\right)^2\label{case3rhs}
    \end{align}
    where the last line uses the assumption on $|\lambda_2|$. The last term of \eqref{case3rhs} can be bounded below by
    \begin{align}
        \frac{\gamma^{-2/3}}{2}g^{ii}\left(\frac{b_i}{\bar{b}}\right)^2 &\geq \frac{\gamma^{-2/3}}{2}g^{22}\left(\frac{b_2}{\bar{b}}\right)^2\nonumber\\
        &\geq  \frac{\gamma^{-2/3}}{2}g^{22}\left(140\frac{y_2^2}{\rho^2} - 72\alpha^2y_2^2\lambda_2^2 - 72\beta^2\Gamma^2\lambda_2^2\right)\nonumber\\
        &\geq  g^{22}\frac{70}{8}\frac{\gamma^{-2/3}}{\rho^2} - \gamma^{-2/3}36\alpha^2 - \gamma^{-2/3}36\beta^2\Gamma^2.\label{case13bisq}
    \end{align}
    Plugging \eqref{case13bisq} into \eqref{case3rhs} and rearranging terms yields
    \begin{align*}
        0\geq &g^{22}\left(\frac{70}{8}\frac{\gamma^{-2/3}}{\rho^2} - \frac{72}{\rho^2} - \frac{24}{\rho}- 1 - \frac{C(\|D^2\Theta\|_{L^\infty(B_1(0))})}{4}\right)\\
        &\quad + \frac{\beta}{2}\cdot\frac{291}{292} -  \frac{C(\|D^2\Theta\|_{L^\infty(B_1(0))})}{\bar{b}^\frac{2}{3}}\\
        &\quad- 36\gamma^{-2/3}(\alpha^2 + \beta^2\Gamma^2)  -\frac{3}{4}(\alpha^\frac{4}{3} + \beta^\frac{4}{3}\Gamma^\frac{4}{3}).
    \end{align*}
    We choose $\alpha,\beta$ so that
    \begin{equation}\label{albet3}
        \alpha^\frac{4}{3} < \frac{\beta}{292}\gamma^\frac{2}{3} \quad\text{and}\quad \beta^\frac{4}{3}\Gamma^\frac{4}{3} < \frac{\beta}{292}\gamma^\frac{2}{3} 
    \end{equation}
    which is consistent with \eqref{albet} and \eqref{albet2}. From \eqref{albet} and \eqref{albet3}, the inequality above becomes
    \begin{align}
        0\geq &g^{22}\left(\frac{70}{8}\frac{\gamma^{-2/3}}{\rho^2} - \frac{72}{\rho^2} - \frac{24}{\rho} - 1 - \frac{C(\|D^2\Theta\|_{L^\infty(B_1(0))})}{4}\right)\nonumber\\
        &\quad+ \beta\left( \frac{1}{2}\cdot\frac{291}{292} - \frac{72}{292} - \frac{2}{292}\right) - \frac{C(\|D^2\Theta\|_{L^\infty(B_1(0))})}{\bar{b}^\frac{2}{3}}\nonumber\\
        = &g^{22}\left(\frac{70}{8}\frac{\gamma^{-2/3}}{\rho^2} - \frac{72}{\rho^2} - \frac{24}{\rho} - 1 - \frac{C(\|D^2\Theta\|_{L^\infty(B_1(0))})}{4}\right)\nonumber\\
        &\quad+ \beta\frac{147}{2\cdot 292} - \frac{C(\|D^2\Theta\|_{L^\infty(B_1(0))})}{\bar{b}^\frac{2}{3}}.\label{case13cont}
        \end{align}
    We may assume that
    \begin{equation}\label{c13barb}
        \bar{b}^\frac{2}{3} > \frac{2\cdot 292 C(\|D^2\Theta\|_{L^\infty(B_1(0))})}{147\beta}
    \end{equation}
    otherwise, $\bar{b}$ is bounded and we would be done.
    
    We now choose $\gamma$ to be small enough so that 
    \begin{align}
        \frac{70}{8}\frac{\gamma^{-2/3}}{\rho^2} &- \frac{72}{\rho^2} - \frac{24}{\rho}- 1 - \frac{C(\|D^2\Theta\|_{L^\infty(B_1(0))})}{4} > 0\nonumber \\
        &\text{i.e.}\quad \gamma< \left(\frac{70}{8(72 + 25 + C(\|D^2\Theta\|_{L^\infty(B_1(0))}))}\right)^\frac{3}{2}.\label{case13gam}
    \end{align}
    Combining \eqref{c13barb} and \eqref{case13gam} we contradict the inequality in \eqref{case13cont}. Hence, $y$ cannot be a maximum of $P$. 
    \end{enumerate}
    
    \textbf{Case 2.} Now, if $-\frac{\pi}{2} < \Theta(y)\leq 0$, it follows that $-\lambda_2>|-\lambda_1|$ and $-\lambda_1 < 0$. We can reduce to case $1$ as follows: Let the test function in case $1$ be denoted by $P_+$. If $P_+$ achieves its maximum at $y$ for the solution $u$, then for $v = -u$,
    \[
    P_- := \nu\ln\rho(x) - \alpha ( x\cdot Dv - v) + \beta|Dv|^2/2 +\ln\max(\bar{b},\gamma^{-1})
    \]
    achieves its maximum at $y$. We remark that $\bar{b}$ stays the same since it is symmetric and even in the eigenvalues of the Hessian. We reduce to the argument in case $1$ via the following substitutions:
    \begin{itemize}
        \item $\alpha$ becomes $-\alpha$
        \item $\Theta(y)$ becomes $-\Theta(x_0)$
        \item $\lambda_1$ becomes $-\lambda_2$
        \item $\lambda_2$ becomes $-\lambda_1$
        \item $\max\{0,\Theta\}$ becomes $\max\{0, -\Theta\}$.
    \end{itemize}
    and the result holds since the argument in case 1 is independent of the sign of $\alpha$.

    \textbf{Case 3.} If $\Theta(y) \geq\frac{\pi}{2}$, it follows that $\lambda_1 >\lambda_2 \geq 0$. We again use $P$ as in case 1, and equation \eqref{gradP2} stays the same while \eqref{HessP2} becomes (where we safely assume $1<\bar{b}/8$)
    \begin{align}
        12\frac{g^{ii}}{\rho} + 24\frac{g^{ii}y_i^2}{\rho^2}&\geq \alpha g^{ii}\lambda_i + \beta g^{ii}\lambda_i^2 + (\alpha y +  \beta Du)\cdot D\Theta \nonumber\\
        &\qquad + \frac{1}{4}\bar{b}g^{ii}\left(\frac{b_i}{\bar{b}}\right)^2 +  g^{ii}\lambda_i\Theta_i \left(\frac{b_i}{\bar{b}}\right) - \frac{C(\|D\Theta\|_{L^\infty(B_1(0))},\|D^2\Theta\|_{L^\infty(B_1(0))})}{\bar{b}}.\label{case3hessP2}
    \end{align}
    This case will be similar to case $1.3$ in that we will derive a contradiction. We bound the left-hand side of \eqref{case3hessP2} by
    \begin{equation}\label{case3lhs}
        12\frac{g^{ii}}{\rho} + 24\frac{g^{ii}y_i^2}{\rho^2} \leq 36\frac{g^{ii}}{\rho^2} \leq 72\frac{g^{22}}{\rho^2}.
    \end{equation}
    For the right hand side, using $\lambda_1 > 1$, we see
    \begin{align}
        \alpha g^{ii}\lambda_i + \beta g^{ii}\lambda_i^2 &\geq \frac{\beta}{2} -\frac{\beta}{8} - g^{22}\frac{2\alpha^2}{\beta}\nonumber\\
        &\geq \frac{3\beta}{8} - g^{22}\frac{2\gamma}{292}\nonumber\\
        &\geq \frac{3\beta}{8} - g^{22}\frac{2}{292}\label{case3alpha}
    \end{align}
    where we use the bound on $\alpha^\frac{4}{3}$ from \eqref{albet4} below since $\alpha^2 < \alpha^\frac{4}{3}$.
    Similar to case $1.3$, we bound
    \begin{equation}\label{case3LB}
        g^{ii}\lambda_i\Theta_i \left(\frac{b_i}{\bar{b}}\right) \geq -\frac{\bar{b}}{8}g^{ii}\left(\frac{b_i}{\bar{b}}\right)^2 - 2\frac{|D\Theta(y)|^2}{\bar{b}}.
    \end{equation}
    Combining \eqref{case3lhs}, \eqref{case3alpha}, \eqref{youngs1} from before, and \eqref{case3LB} into \eqref{case3hessP2}, we get   
    \begin{align}
        72\frac{g^{22}}{\rho^2}&\geq \frac{3\beta}{8} - g^{22}\frac{2}{292} -g^{22}\frac{C(\|D^2\Theta\|_{L^\infty(B_1(0))})}{4}\nonumber\\
        &\quad  \frac{\gamma^{-1}}{8}g^{ii}\left(\frac{b_i}{\bar{b}}\right)^2 - \frac{3}{4}(\alpha^\frac{4}{3}+ \beta^\frac{4}{3}\Gamma^\frac{4}{3})  - \frac{C(\|D\Theta\|_{L^\infty(B_1(0))},\|D^2\Theta\|_{L^\infty(B_1(0))})}{\bar{b}}.\label{case3full}
    \end{align}
    Using the bound \eqref{case13bisq} but with a different $\gamma$ coefficient, we see that
    \begin{equation}\label{case3bisq}
        \frac{\gamma^{-1}}{8}g^{ii}\left(\frac{b_i}{\bar{b}}\right)^2 \geq g^{22}\frac{70}{32}\frac{\gamma^{-1}}{\rho^2} - \gamma^{-1}9(\alpha^2 + \beta^2\Gamma^2).
    \end{equation}
    Plugging \eqref{case3bisq} into \eqref{case3full} and rearranging terms, we get
    \begin{align*}
        0&\geq g^{22}\left(\frac{70}{32}\frac{\gamma^{-1}}{\rho^2}  -\frac{72}{\rho^2} -\frac{2}{292} -\frac{C(\|D^2\Theta\|_{L^\infty(B_1(0))})}{4}\right)\\
        &\quad + \frac{3\beta}{8} -\gamma^{-1}9(\alpha^2 + \beta^2\Gamma^2) - \frac{3}{4}(\alpha^\frac{4}{3}+ \beta^\frac{4}{3}\Gamma^\frac{4}{3}) -  \frac{C(\|D\Theta\|_{L^\infty(B_1(0))},\|D^2\Theta\|_{L^\infty(B_1(0))})}{\bar{b}}.
    \end{align*}
    Choosing
    \begin{equation}\label{albet4}
        \alpha^\frac{4}{3} < \frac{\beta}{292}\gamma \quad\text{and}\quad \beta^\frac{4}{3}\Gamma^\frac{4}{3} < \frac{\beta}{292}\gamma,
    \end{equation}
    which is consistent with \eqref{albet}, \eqref{albet2}, and \eqref{albet3}, we get
    \begin{align*}
         0&\geq g^{22}\left(\frac{70}{32}\frac{\gamma^{-1}}{\rho^2}  -\frac{72}{\rho^2}-\frac{2}{292}-\frac{C(\|D^2\Theta\|_{L^\infty(B_1(0))})}{4} \right)\\
         &\quad+ \beta\left(\frac{3}{8}-\frac{18}{292} - \frac{3}{2\cdot292}\right) -  \frac{C(\|D\Theta\|_{L^\infty(B_1(0))},\|D^2\Theta\|_{L^\infty(B_1(0))})}{\bar{b}}\\
         &= g^{22}\left(\frac{70}{32}\frac{\gamma^{-1}}{\rho^2}  -\frac{72}{\rho^2}-\frac{2}{292} --\frac{C(\|D^2\Theta\|_{L^\infty(B_1(0))})}{4}\right)\\
         &\quad+ \beta\frac{45}{146} -  \frac{C(\|D\Theta\|_{L^\infty(B_1(0))},\|D^2\Theta\|_{L^\infty(B_1(0))})}{\bar{b}}.
    \end{align*}
    We may assume that 
    \[
    \bar{b} > \frac{146 C(\|D\Theta\|_{L^\infty(B_1(0))},\|D^2\Theta\|_{L^\infty(B_1(0))})}{45\beta}
    \]
    otherwise $\bar{b}$ is bounded, and we would be done.
    
    Lastly, choosing
    \[
    \gamma < \frac{32}{70(73 + C(\|D^2\Theta\|_{L^\infty(B_1(0))}))}
    \]
    we derive a contradiction that $y$ was a maximum point of $P$.

    \textbf{Case 4.} Lastly, we will assume $-\pi<\Theta(y) \leq -\frac{\pi}{2}$, from which it follows that $\lambda_2 < \lambda_1 < 0$. Now the assumption that $P>>1$ implies that $-\lambda_2 >> 1$. We use the reflection argument (see case 2) applied to case 3, which is proven independent of the sign of $\alpha$.
    \end{proof}
    We remark that the choices of the constants in \eqref{albet}, \eqref{albet2}, \eqref{albet3}, and \eqref{albet4}, are all consistent. The constants $\gamma$ and $\Gamma$ are fixed constants determined by $\Theta$ and $\|u\|_{C^1(B_1(0))}$ respectively. We first choose $\alpha$ and $\beta$ such that \eqref{albet2} and the left inequality of \eqref{albet4} are satisfied:
    \begin{equation}\label{ogalbet}
        \frac{\alpha^\frac{4}{3}}{\beta}< \frac{\gamma}{292} \quad\text{and}\quad 16\Gamma < \frac{|\alpha|}{\beta}.
    \end{equation}
    We achieve this by choosing $\alpha$ and $\beta$ such that $\alpha^\frac{4}{3} < \beta < |\alpha|$. The inequalities in \eqref{ogalbet} are enough to prove that the remaining conditions in \eqref{albet}, \eqref{albet3}, and \eqref{albet4} will hold. From the two inequalities in \eqref{ogalbet}, we automatically get
    \[
    \beta^\frac{4}{3}\Gamma^\frac{4}{3} < \frac{\alpha^\frac{4}{3}}{16^\frac{4}{3}}< \frac{\beta}{16^\frac{4}{3}292}\gamma < \frac{\beta}{292}\gamma < \frac{\beta}{292}\gamma^\frac{2}{3}
    \]
    which proves the right inequality in \eqref{albet4} and further implies \eqref{albet3}. We next observe that
    \[
    \alpha^2 = \left(\alpha^\frac{4}{3}\right)^\frac{3}{2} < \frac{\beta^\frac{3}{2}}{292^\frac{3}{2}}\gamma^\frac{3}{2} < \frac{\beta}{292}\gamma 
    \]
    and the right inequality in \eqref{ogalbet} together imply \eqref{albet}.   

\section{Alexandrov-type Theorem}\label{sec_alexandrv}
We seek to prove the following theorem, which is inspired by \cite{RY3}.
\begin{theorem}\label{alexandrov}
    Let $u_k$ be a sequence of smooth solutions to \eqref{s} for some $\Theta_k$ on $B_R(0)$. Suppose that $u_k$ converges uniformly to $u$, a viscosity solution of \eqref{s} on $B_R(0)$, where the sequence of phases $\Theta_k$ satisfy $\|D\Theta_k\|_{L^\infty(B_R(0))},\|D^2\Theta_k\|_{L^\infty(B_R(0))}< K_2$ and $\Theta_k \to \Theta$ in $C^{1,\alpha}$ for $\alpha \in (0,1)$. Then $u$ is twice differentiable almost everywhere in $B_R(0)$, i.e., for any $x\in B_R(0)$, there exists a quadratic polynomial $Q$ such that
    \[
    \sup_{y\in B_r(x)}|u(y)-Q(y)|= o(r^2).
    \]
\end{theorem}

To prove this, it suffices to show that for the viscosity solution $u$, we get
\begin{align}
    \sup_{x,y \in \Omega: x\neq y}d_{x,y}^{n+1}\frac{|u(x)-u(y)|}{|x-y|}&\leq C(n)\int_\Omega |u|dx\label{loclip}\\
    \lim_{r\to 0}\fint_{B_r(x)}|Du(y) - Du(x)|dy &= 0\quad\text{a.e. }x\in B_R(0)\label{weakD1}\\
    \lim_{r\to 0}\fint_{B_r(x)}|D^2u(y) - D^2(x)|dy &= 0 \quad\text{a.e. }x\in B_R(0)\label{weakD2}\\
    \lim_{r\to 0} \frac{1}{r^n}\|[D^2u]_s\|(B_r(x)) &= 0\quad\text{a.e. }x\in B_R(0)\label{lebesguesing}.
\end{align}
The first line states that $u$ is locally Lipschitz. For more general operators, we actually need that $u - L$, where $L$ is any linear function, satisfies \eqref{loclip}; however, in this case, if $u$ solves \eqref{s}, then so does $u - L$ and \eqref{loclip} will hold true for $u-L$. The second line states that $u$ is differentiable almost everywhere with the weak gradient $Du \in L^\infty_{loc}$. The third and fourth lines state the almost everywhere existence of $D^2u$ as a Radon measure and the Lebesgue decomposition of that measure. If the four conditions above are satisfied, the proof of Theorem \ref{alexandrov} follows exactly as in \cite{RY3},\cite{EG}, or \cite{CT}.

In order to prove \eqref{loclip}-\eqref{lebesguesing}, we need the following two lemmas.

\subsection{Scale Invariant Gradient Estimate}
This follows by adopting a similar strategy to \cite{BMS}. We will include details here for the convenience of the readers using notation from \cite{BMS}.
\begin{lemma}\label{gradlem}
    Let $u$ be a $C^3(\overline{B_R(0)})$ solution of \eqref{s} on $B_R(0)\subset \re^2$, where $\Theta \in C^2(B_R(0))$ satisfies $-\pi < \Theta < \pi$ and $D\Theta = 0$ on the level set $\{\Theta = 0\}$. Then
    \begin{equation}\label{gradest}
        R|Du(0)|\leq C(n,\|D^2\Theta\|_{L^\infty(B_R(0))})(\osc_{B_R(0)}u)(1+\osc_{B_R(0)}u).  
    \end{equation}
\end{lemma}
\begin{proof}
    The proof in \cite{BMS} works with the following modifications:
    \begin{equation*}
        v(x) = \frac{u(Rx)}{R^2} \quad\text{hence}\quad Dv(x) = \frac{Du(Rx)}{R} \text{ and } D^2v(x) = D^2u(Rx)
    \end{equation*}
    for $x\in B_1(0)$. This last equation shows that $v$ satisfies the equation $F(D^2v(x)) = \Theta(Rx)$. As the test function
    \[
    w = \frac{1}{1-|x|^2}|Dv| + Av^2/2
    \]
    is an even function of $v$, we are able to use the reflection $u\to -u$, for the case where $\Theta(Rx_0) < 0$.
    
    We define
    \begin{equation*}
        M_R = \osc_{B_R(0)}u = R^2\osc_{B_1(0)}v = R^2M \text{ for } M=\osc_{B_1(0)}v.
    \end{equation*}
    The rest of the proof remains the same until we reach \cite[(2.11)]{BMS}, where we now get
    \begin{equation}\label{gradv}
        (1-|x_0|^2)|Dv|\leq C(n)M(1 + (1-|x_0|^2)R|D\Theta(Rx_0)\|\lambda_n|)
    \end{equation}
    where an extra factor of $R$ comes from taking the derivative of $\Theta$. If $\Theta(Rx_0) >0$, then $\lambda_n = \lambda_2$ and if $\Theta(Rx_0) < 0$, then $\lambda_n = \lambda_1$. From the assumption on $\Theta$, if $\Theta(Rx_0)> 0$, then we can apply the pointwise interpolation inequality \cite[Lemma 7.7.2]{lhorm1} to $\max\{0,\Theta\}$. If $\Theta(Rx_0) = 0$, then $|D\Theta(Rx_0)| = 0$ and the proof would be done. Overall, we obtain
    \[
    |D\Theta(Rx_0)|^2\leq \frac{\Theta(Rx_0)^2}{1-|x_0|^2} +2\|D^2\Theta\|_{L^\infty(B_R(0))}|\Theta(Rx_0)|.
    \]
    We also recall from \cite{BMS} that
    \begin{equation*}
        |\Theta| \leq \frac{1}{|\lambda_n|} \text{ and } 1-|x_0|^2 \sim \frac{|Dv|}{|\lambda_n|}.
    \end{equation*}
    Substituting these into \eqref{gradv} we get
    \[
    M^{-1}(1-|x_0|^2)|Dv|\lesssim_\Theta 1 + R((1-|x_0|^2)|Du|)^\frac{1}{2}
    \]
    from which we get
    \[
    (1-|x_0|^2)|Dv|\lesssim_\Theta M + R^2M^2.
    \]
    Altogether, we get
    \begin{align*}
        \frac{1}{R}|Du(0)| = |Dv(0)|&\leq (1-|x_0|^2)|Dv(x_0)| + \frac{3\sqrt{n}v^2(x_0)}{2M}\\
        &\leq C(n,\|D^2\Theta\|_{L^\infty(B_R(0))})(M + R^2M^2) + C(n)M\\
        &= C(n,\|D^2\Theta\|_{L^\infty(B_R(0))})(\frac{M_R}{R^2} + \frac{M_R^2}{R^2}) + C(n)\frac{M_R}{R^2}.
    \end{align*}
    Hence,
    \[
    R|Du(0)| \leq C(n,\|D^2\Theta\|_{L^\infty(B_R(0))})(\osc_{B_R(0)}u)(1 + \osc_{B_R(0)}u).
    \]
\end{proof}

\subsection{Volume bound by the gradient}
Here we will show how to bound the volume form. 

\begin{lemma}
    Let $u$ be a smooth solution of \eqref{s} on $B_{2R}(0)\subset\re^2$, where $\Theta \in C^2(B_{2R}(0))$ satisfies $-\pi < \Theta < \pi$. Then
    \begin{equation}\label{volbd}
        \frac{1}{|B_R|}\int_{B_R(0)} Vdx \leq C\left(1+(1+1/R^2)\|Du\|_{L^\infty(B_{2R}(0))}^2 +(1+1/R)\|Du\|_{L^\infty(B_{2R}(0))}\right)
    \end{equation}
    where $C$ depends on the dimension $2$, $\|D\Theta\|_{L^\infty(B_{2R}(0))}$, and $\|D^2\Theta\|_{L^\infty(B_{2R}(0))}$.
\end{lemma} 

\begin{proof}
We first note that the volume element can be written as
\begin{align*}
    V^2 &= (1 + \lambda_1^2)(1 + \lambda_2^2)\\
    &= 1 + \lambda_1^2 + \lambda_2^2 + \sigma_2^2\\
    &=1 + \sigma_1^2 - 2\sigma_2 + \sigma_2^2\\
    &= \sigma_1^2 + (1-\sigma_2)^2.
\end{align*}
Using the alternate form of the equation \eqref{tanpde}, it follows that
\begin{align}
    V &= \sec|\Theta\|(1-\sigma_2)|\label{Vsec}\\
    V &= \csc|\Theta\|\sigma_1|\label{Vcsc}
\end{align}
where the sign of the above pieces depends on the range and sign of $\Theta$, which dictates the signs of $\sigma_1$ and $\sigma_2$.

We will use the product of two cutoff functions over different regions. Let $0\leq\chi\leq 1$ be a smooth cutoff function supported in $B_2(0)$ where $\chi = 1$ on $B_1(0)$, $|D\chi| < 1$ and $|D^2\chi|< 2$. We will define the following five cutoff functions $\rho_i$ over the different ranges of $\Theta$:
\begin{align*}
    0\leq\rho_1\leq 1 &\text{ supported on }  (-9\pi/8, -5\pi/8) \text{ where } \rho_1 = 1 \text{ on } (-\pi, -3\pi/4)\\
    0\leq\rho_2\leq 1 &\text{ supported on } (-7\pi/8, -\pi/8) \text{ where } \rho_2 = 1 \text{ on } (-3\pi/4, -\pi/4)\\
    0\leq\rho_3\leq 1 &\text{ supported on } (-3\pi/8, 3\pi/8) \text{ where } \rho_3 = 1 \text{ on } (-\pi/4, \pi/4)\\
    0\leq\rho_4\leq 1 &\text{ supported on } (\pi/8, 7\pi/8) \text{ where } \rho_4 = 1 \text{ on } (\pi/4, 3\pi/4)\\
    0\leq\rho_5\leq 1 &\text{ supported on } (5\pi/8,9\pi/8) \text{ where } \rho_5 = 1 \text{ on } (3\pi/4,\pi)
\end{align*}
where each $\rho_j$ satisfies $|\rho_j'|\leq \frac{8}{\pi}, |\rho_j''|\leq \frac{64}{\pi^2}$.

We observe that
\[
\int_{B_R(0)}Vdx \leq \sum_{j=1}^5\int_{B_{R}(0)}\chi(x/R)\rho_j(\Theta(x))Vdx
\]
since $\chi(x/R) = 1$ on $B_R(0)$ and $1\leq \sum_{j=1}^5\rho_j$ on $\Theta(B_R(0))$.
Looking at each component, we see
\begin{align}
    &\int_{B_{R}(0)}\chi(x/R)\rho_1(\Theta(x))|\sec\Theta(1-\sigma_2)| dx \leq C_1\int_{B_{2R}(0)}\chi(x/R)\rho_1(\Theta(x))(\sigma_2 - 1)dx \nonumber\\
    &\int_{B_{R}(0)}\chi(x/R)\rho_2(\Theta(x))|\csc\Theta\sigma_1 |dx \leq C_1\int_{B_{2R}(0)}\chi(x/R)\rho_2(\Theta(x))(-\sigma_1) dx\nonumber\\
    &\int_{B_{R}(0)}\chi(x/R)\rho_3(\Theta(x))|\sec\Theta(1-\sigma_2)| dx \leq C_1\int_{B_{2R}(0)}\chi(x/R)\rho_3(\Theta(x))(1-\sigma_2) dx \label{sig2}\\
    &\int_{B_{R}(0)}\chi(x/R)\rho_4(\Theta(x))|\csc\Theta\sigma_1| dx \leq C_1\int_{B_{2R}(0)}\chi(x/R)\rho_4(\Theta(x))\sigma_1 dx\label{sig1}\\
    &\int_{B_{R}(0)}\chi(x/R)\rho_5(\Theta(x))|\sec\Theta(1-\sigma_2)| dx \leq C_1\int_{B_{2R}(0)}\chi(x/R)\rho_5(\Theta(x))(\sigma_2-1) dx\nonumber
\end{align}
where we first use the cutoffs $\rho_j$, and either \eqref{Vsec} or \eqref{Vcsc}, as well as the bounds on 
\[
|\sec\Theta|,|\csc\Theta| \leq \frac{2}{\sqrt{2-\sqrt{2}}} =: C_1
\]
in their respective regions on the supports of $\rho_j$. Then we increase the radius to $B_{2R}(0)$ using the $\chi$ cutoff and the fact that all of the integrands are nonnegative. We will show how to bound the integrals in \eqref{sig2} and \eqref{sig1} using integration by parts, as the other three integrals are similar.

The term with $\sigma_1$ is bounded by
\begin{align*}
    \int_{B_{2R}(0)}\chi(x/R)&\rho_4(\Theta(x))\sigma_1 dx = \int_{B_{2R}(0)}\chi(x/R)\rho_4(\Theta(x))\text{div}Du \;dx\nonumber\\
    &= -\frac{1}{R}\int_{B_{2R}(0)}\rho_4(\Theta(x))\langle D\chi, Du\rangle dx  \nonumber\\
    &\quad -\int_{B_{2R}(0)}\chi(x/R)\rho_4'(\Theta)\langle D\Theta,Du\rangle dx \nonumber\\
    &\leq 4|B_{R}|\left(\frac{1}{R} + \frac{8}{\pi}\|D\Theta\|_{L^\infty(B_{2R}(0))}\right)\|Du\|_{L^\infty(B_{2R}(0))}.
\end{align*}

To bound the integral with $\sigma_2$, we can use the inductive argument from \cite{WaY} and the divergence structure of $\sigma_k(D^2u)$:
\begin{align*}
    k\sigma_k(D^2u) &= \sum_{i,j=1}^2\frac{\pd \sigma_k}{\pd u_{ij}}\frac{\pd^2u}{\pd x_i\pd x_j} = \sum_{i,j=1}^2\frac{\pd}{\pd x_i}\left(\frac{\pd \sigma_k}{\pd u_{ij}}\frac{\pd u}{\pd x_j}\right)\\
    &= \text{div}(L_{\sigma_k}Du),
\end{align*}
where $L_{\sigma_k}$ is the matrix given by $\displaystyle\left(\frac{\pd\sigma_k}{\pd u_{ij}}\right)$. We get
\begin{align*}
    \int_{B_{2R}(0)}\chi(x/R)&\rho_3(\Theta(x))(-\sigma_2)dx =-\frac{1}{2} \int_{B_{2R}(0)}\chi(x/R)\rho_3(\Theta(x))\text{ div}(L_{\sigma_2}Du)dx\nonumber\\
    &= \frac{1}{2}\int_{B_{2R}(0)}\langle D(\chi(x/R)\rho_3(\Theta(x))),L_{\sigma_2}Du\rangle dx\nonumber\\
    &= \frac{1}{2}\int_{B_{2R}(0)}\pd_1\big[\chi(x/R)\rho_3(\Theta(x))\big]\left[\pd_2(u_1u_2) - \pd_1(u_2^2)\right] dx\nonumber\\
    &\;+ \frac{1}{2}\int_{B_{2R}(0)}\pd_2\big[\chi(x/R)\rho_3(\Theta(x))\big]\left[\pd_1(u_1u_2) -\pd_2(u_1^2)\right]dx\nonumber\\
    &= \frac{1}{2}\int_{B_{2R}(0)} \pd_1^2\big[\chi(x/R)\rho_3(\Theta(x))\big]u_2^2\;dx\nonumber\\
    &+\frac{1}{2}\int_{B_{2R}(0)} \pd_2^2\big[\chi(x/R)\rho_3(\Theta(x))\big]u_1^2\;dx\nonumber\\
    &-\frac{1}{2}\int_{B_{2R}(0)} 2\pd_{12}^2\big[\chi(x/R)\rho_3(\Theta(x))\big]u_1u_2\;dx\nonumber\\
    &\leq 4|B_R|\left(\frac{2}{R^2} + \frac{48}{\pi^2}\|D\Theta\|_{L^\infty(B_{2R}(0))}^2 + \frac{4}{\pi}\|D^2\Theta\|_{L^\infty(B_{2R}(0))}\right)\|Du\|_{L^\infty(B_{2R}(0))}^2
\end{align*}
where the third line uses the divergence structure of $L_{\sigma_2}Du$ and integration by parts. To see this, observe that
\[
\sigma_2 = \det\begin{bmatrix}
    u_{11} & u_{12}\\
    u_{21} & u_{22}
\end{bmatrix} = u_{11}u_{22}-u_{12}u_{21}
\]
from which we get
\[
L_{\sigma_2} = \begin{bmatrix}
    u_{22} & -u_{21}\\
    -u_{12} & u_{11}
\end{bmatrix}.
\]
Altogether this becomes
\begin{align*}
    L_{\sigma_2}Du &= (u_{22}u_{1} - u_{21}u_2, -u_{12}u_1 + u_{11}u_2)\\
    &= ( \pd_2(u_2u_1) - 2u_2u_{21}, \pd_1(u_1u_2) - 2u_{12}u_1)\\
    &= (\pd_2(u_2u_1) - \pd_1(u_2^2),\pd_1(u_1u_2) - \pd_2(u_1^2)).
\end{align*}
The bound on the volume element is then
\begin{equation*}
    \frac{1}{|B_R|}\int_{B_R(0)}Vdx \leq 12C_1\left(1+C_2\|Du\|_{L^\infty(B_{2R}(0))}^2 + C_3\|Du\|_{L^\infty(B_{2R}(0))}\right).
\end{equation*}
where
\[
C_2 = \frac{2}{R^2} + \frac{48}{\pi^2}\|D\Theta\|^2_{L^\infty(B_{2R}(0))} + \frac{4}{\pi}\|D^2\Theta\|_{L^\infty(B_{2R}(0))}\quad\text{and} \quad C_3 = \frac{1}{R} + \frac{8}{\pi}\|D\Theta\|_{L^\infty(B_{2R}(0))}.
\]
\end{proof}

\begin{enumerate}
    \item[Proof of \eqref{loclip}] Each $u_k$ satisfies \eqref{gradest}. This is similar to \cite[Theorem 3.3]{TW2}. Using the norm definitions \cite[(2.11)]{TW2} and interpolation \cite[Lemma 2.6]{TW2}, we can achieve the estimate similar to \cite[Corollary 3.4]{TW2} which is exactly \eqref{loclip}. This holds for all $u_k$. The right-hand side can then be bounded by
    \[
    \int_{\Omega}|u_k|dx \leq \int_{\Omega}|u_k - u|dx +\int_{\Omega}|u|dx.
    \]
    As $u_k\to u$, we get \eqref{loclip} holds for $u$ as well. \hfill\qedsymbol

    \item[Proof of \eqref{weakD1}] We now prove that $u$ is weakly differentiable almost everywhere. We will prove that the weak derivatives will be given by some function $f_i = \lim_{k\to \infty} \frac{\pd u_k}{\pd x_i}$ using the gradient estimate \eqref{gradest}. Using integration by parts, we have that for any $\phi\in C_c^\infty(B_R(0))$
    \begin{equation}\label{intbyparts1}
        \int_{B_R(0)} u_k \frac{\pd \phi}{\pd x_i}dx = -\int_{B_R(0)}\phi \frac{\pd u_k}{\pd x_i}dx.
    \end{equation}
    Because $u_k \to u$ uniformly, we get that the left hand side of \eqref{intbyparts1} becomes
    \begin{equation}\label{weakd1lhs}
        \lim_{k\to \infty}\int_{B_R(0)} u_k\frac{\pd \phi}{\pd x_i}dx = \int_{B_R(0)}\lim_{k\to \infty}u_k\frac{\pd \phi}{\pd x_i}dx = \int_{B_R(0)}u\frac{\pd \phi}{\pd x_i}dx.
    \end{equation}
    
    To bound the right-hand side, we use the gradient estimate:
    \[
    \left|\phi \frac{\pd u_k}{\pd x_i} \right|\leq |\phi\|Du_k|\leq |\phi|C(n,K_2)(\osc_{B_R(0)}u_k)(1 + \osc_{B_R(0)}u_k).
    \]
    By the variant of the Dominated Convergence Theorem \cite[Chapter 1.3, Theorem 4]{EG}, we get that
    \begin{equation}\label{weakd1rhs}
        \lim_{k\to \infty}\int_{B_R(0)}\phi \frac{\pd u_k}{\pd x_i}dx = \int_{B_R(0)}\phi \lim_{k\to \infty}\frac{\pd u_k}{\pd x_i} dx = \int_{B_R(0)}\phi f_i dx.
    \end{equation}
    By combining \eqref{weakd1lhs} and \eqref{weakd1rhs}, we get exactly
    \[
    \int_{B_R(0)}u\frac{\pd \phi}{\pd x_i}dx = \int_{B_R(0)}\phi f_i dx
    \]
    as desired. \hfill\qedsymbol
    
    \item[Proof of \eqref{weakD2}/\eqref{lebesguesing}] We now prove that $u$ is weakly second differentiable almost everywhere. Similar to before, this will come from proving that the weak second derivatives are given by $f_{ij} = \lim_{k\to \infty}\frac{\pd^2 u_k}{\pd x_i\pd x_j}$.
    Integrating by parts we see that for any $\phi\in C_c^\infty(B_R(0))$,
    \begin{equation}\label{intbyparts2}
        \int_{B_R(0)}u_k\frac{\pd^2\phi}{\pd x_i\pd x_j}dx = \int_{B_R(0)}\phi \frac{\pd^2 u_k}{\pd x_i\pd x_j}dx.
    \end{equation}
    Again since $u_k \to u$ uniformly, we see that the left hand side of \eqref{intbyparts2} becomes
    \begin{equation}\label{weakd2lhs}
        \lim_{k\to \infty}\int_{B_R(0)} u_k\frac{\pd^2\phi}{\pd x_i\pd x_j}dx = \int_{B_R(0)}\lim_{k\to \infty}u_k\frac{\pd^2\phi}{\pd x_i\pd x_j}dx = \int_{B_R(0)}u\frac{\pd^2\phi}{\pd x_i\pd x_j}dx.
    \end{equation}
    The right-hand side is trickier since we don't have a point-wise bound on the second derivatives. However, we do have the following pointwise
    \[
    \left|\phi \frac{\pd^2 u_k}{\pd x_i\pd x_j}\right|\leq |\phi|\sqrt{\det(I_2 + (D^2u_k)^2)} := |\phi|V_k.
    \]
    Furthermore, using the volume bound \eqref{volbd}, Young's inequality, and the fact that $u_k\to u$ uniformly, we get
    \begin{align*}
        \int_{B_R(0)} V_kdx &\leq C(R,\|D\Theta_k\|_{L^\infty(B_{2R(0)})},\|D^2\Theta_k\|_{L^\infty(B_{2R(0)})})(\|Du_k\|^2_{L^\infty(B_{2R(0)})} + \|Du_k\|_{L^\infty(B_{2R(0)})} + 1)\\
        &\leq C(R,K_2)((\osc_{B_{2R(0)}}u_k)^4 + 1))\\
        &\leq C(R,K_2)( (\osc_{B_{2R(0)}}(u_k - u))^4 + (\osc_{B_{2R(0)}}u)^4 + 1)\\
        &\leq C(R,K_2)( \max_{k\leq M}\{(\osc_{B_{2R(0)}}(u_k - u))^4\} + (\osc_{B_{2R(0)}}u)^4 + 2)
    \end{align*}
    where the last line is independent of $k$ as we choose $M$ large enough so that $\osc_{B_R(0)}(u_k - u)\leq 1$ for all $k\geq M$. Using \cite[Theorem 4.5.6]{BogachevMT}, we have that there exists some $V$ such that
    \[
    \lim_{k\to\infty}\int_{B_R(0)}|\phi|V_kdx = \int_{B_R(0)}|\phi|Vdx
    \]
    and hence, by the variant of the Dominated Convergence Theorem \cite[Chapter 1.3, Theorem 4]{EG}, we get 
    \begin{equation}\label{weakd2rhs}
        \lim_{k\to \infty}\int_{B_R(0)}\phi \frac{\pd^2 u_k}{\pd x_i\pd x_j}dx = \int_{B_R(0)}\phi\lim_{k\to \infty} \frac{\pd^2 u_k}{\pd x_i\pd x_j}dx = \int_{B_R(0)}\phi f_{ij}dx.
    \end{equation}
    By combining \eqref{weakd2lhs} and \eqref{weakd2rhs} we get
    \[
    \int_{B_R(0)}u\frac{\pd^2\phi}{\pd x_i\pd x_j}dx = \int_{B_R(0)}\phi f_{ij}dx
    \]
    as desired.\hfill\qedsymbol
\end{enumerate}

\section{Proof of the Main Theorem}\label{sec_main}
\begin{proof} Our proof follows in four steps.

\begin{enumerate}
    \item[Step 1.] We claim that $|D^2u(0)|$ is bounded above by $\|u\|_{C^1(B_R(0))}$. Otherwise, there exists a sequence of smooth solutions to \eqref{s} with variable $\Theta_k \in C^{1,1}$ on $B_R(0)$ such that $\|u_k\|_{C^1(B_{r_1}(0))}\leq K_1$ for $r_1< R$, but $|D^2u_k(0)|\to\infty$. We note that the gradient and Hessian of $\Theta_k$ are uniformly bounded above by some constant $K_2$. By Arzel\`a-Ascoli, we can extract a subsequence, still denoted $u_k$, which converges uniformly over a smaller ball $B_{r_1}(0)$ and $\Theta_k \to \Theta$ in $C^{1,\alpha}$ for $0<\alpha < 1$. By the closedness of viscosity solutions \cite{CC}, $u_k$ converges uniformly to a continuous viscosity solution, which we denote by $u$. 
    \item[Step 2.] By the Alexandrov-type Theorem \ref{alexandrov}, we have that $u$ is second differentiable almost everywhere on $B_{r_1}(0)$. Fix such a point $y\in B_{r_2}(0)$ where $r_2< r_1/2$ and let $Q(x)$ be the quadratic such that $\sup_{x\in B_{r_2}(y)}|u(x)-Q(x)| = o(r_2^2)$.
    \item[Step 3.] We now apply Fan's \cite{ZF} extension of Savin's \cite{Savin} perturbation theorem  to $v_k = u_k - Q$. We rescale the functions $v_k$ to 
    \[
    \bar{v}_k(\bar{x}) = \frac{1}{r_2^2}v_k(r_2\bar{x} + y) \quad\text{ for } \bar{x}\in B_1(0).
    \]
    It follows that
    \begin{align}
        \|\bar{v}_k\|_{L^\infty(B_1(0))} &\leq \frac{\|u_k(r_2\bar{x} + y)-u(r_2\bar{x} + y)\|_{L^\infty(B_1(0))}}{r_2^2} + \frac{\|u(r_2\bar{x}+y) - Q(r_2\bar{x} + y)\|_{L^\infty(B_1(0))}}{r_2^2}\nonumber\\
        &= \frac{\|u_k-u\|_{L^\infty(B_{r_2}(y))}}{r_2^2} + \frac{\|u - Q\|_{L^\infty(B_{r_2}(y))}}{r_2^2}\nonumber\\
        &\leq \frac{o(k)}{r_2^2 k} + \frac{o(r_2^2)}{r_2^2}\label{savinC1}.
    \end{align}
    We also have that $\bar{v}_k$ solves 
    \begin{equation*}
        G(D^2\bar{w}(\bar{x})) = \sum_{i=1}^2\left[\arctan\lambda_i(D^2\bar{w} + D^2Q) - \arctan\lambda_i(D^2Q) \right] = f_k(\bar{x})
    \end{equation*}
    where
    \[
    f_k(\bar{x}) = \Theta_k(r_2\bar{x} + y) - \sum_{i=1}^2\arctan\lambda_i(D^2Q).
    \]
    We note that as $Q$ is a quadratic, $D^2Q$ is a constant matrix. We then see that $G$ satisfies the hypotheses of \cite[Theorem 1.7]{ZF}, namely, 1) $G$ is elliptic and furthermore, uniformly elliptic near the origin, 2) $G(0) = 0$, and 3) $G$ is $C^2$ near the origin. 

    We show that for $k$ large enough and $r_2$ small enough,
    \begin{equation*}
        \|\bar{v}_k\|_{L^\infty(B_1(0))}\leq c_1 \quad \text{and}\quad \|f_k\|_{C^{0,\alpha}(B_1(0))} \leq c_1
    \end{equation*}
    for some $c_1$ depending on $n,\alpha,Q$.
    We first choose $r_2$ small enough and $k$ large enough so that \eqref{savinC1} is smaller than $c_1$.  We then see that
    \begin{align}
        \|f_k\|_{C^{0,\alpha}(B_1(0))} &= \sup_{\bar{x},\bar{x}'\in B_1(0):\bar{x}\neq \bar{x}'}\frac{|f(\bar{x}) - f(\bar{x}')|}{|\bar{x}-\bar{x}'|^\alpha} \nonumber\\
        &= \sup_{\bar{x},\bar{x}'\in B_1(0):\bar{x}\neq \bar{x}'}\frac{|\Theta_k(r_2\bar{x}+y) - \Theta_k(r_2\bar{x}'+y)|}{|\bar{x}-\bar{x}'|^\alpha}\nonumber\\
        &= \sup_{x,x'\in B_{r_2}(y):x\neq x'}\frac{|\Theta_k(x) - \Theta_k(x')|}{|x-x'|^\alpha}r_2^\alpha\nonumber\\
        &\leq \|\Theta_k\|_{C^{0,\alpha}(B_{r_1}(0))} r_2^{\alpha}\nonumber\\
        &\leq (\|\Theta_k -\Theta\|_{C^{0,\alpha}(B_{r_1}(0))} + \|\Theta\|_{C^{0,\alpha}(B_{r_1}(0))})r_2^{\alpha}\nonumber\\
        &\leq (1 + \|\Theta\|_{C^{0,\alpha}(B_{r_1}(0))})r_2^{\alpha}\label{falpha}
    \end{align}
    where in the last line we choose $k$ large enough that $\|\Theta_k -\Theta\|_{C^{0,\alpha}(B_{r_1}(0))}<1$. Now, we also choose $r_2$ small enough so that \eqref{falpha} is also smaller than $c_1$.
    
    We then get $\bar{v}_k\in C^{2,\alpha}(B_{1/2}(0))$ with
    \begin{equation*}
        \|\bar{v}_k\|_{C^{2,\alpha}(B_{1/2}(0))} \leq C
    \end{equation*}
    which is equivalent to
    \begin{equation}\label{c2al}
        \|u_k\|_{C^{2,\alpha}(B_{r_2/2}(y))}\leq C(n,Q,\alpha,r_2).
    \end{equation}
    
    \item[Step 4.] We now apply the doubling inequality in Lemma \ref{doubling} to \eqref{c2al} and see that
    \begin{align*}
        |D^2u_k(0)|&\leq \sup_{B_{r_2}(y)}|D^2u_k|\\
        &\leq C(r_2,n,\|u_k\|_{C^{1}(B_{r_1}(0))},\|D\Theta_k\|_{L^\infty(B_{r_1}(0))},\|D^2\Theta_k\|_{L^\infty(B_{r_1}(0))},\sup_{B_{r_2/2}(y)}|D^2u|)\\
        &\leq C(r_2,n,K_1,K_2,Q,\alpha)
    \end{align*}
    which contradicts the blowup assumption. This concludes the proof.
\end{enumerate}
\end{proof}

\section*{Appendix} \label{sec_appendix}
The main idea behind proving \eqref{lambda2} and \eqref{lambda1} is the following lemma.
\begin{lemma}\label{MVT}
    Let $f:[x_0,x_1] \to \re$ be continuous and differentiable on $(x_0,x_1)$. If $f(x_0) = 0 $ and $f'(x) \geq 0$ for all $x_0< x < x_1$, then $f(x) \geq f(x_0)$ for all $x_0\leq x < x_1$.
\end{lemma}
\begin{proof}
    This follows directly from the Mean Value Theorem. 
\end{proof}

\subsection{Proof of Equation \eqref{lambda2}}
    In this case, we are taking $x = 1/|\lambda_2|$, so we will consider $0\leq x \leq 1$. We take
    \[
    f(x) = \arctan(x) - \arctan(1)x.
    \]
    Observe that $f(0) = 0$ and $f(1)=0$. We observe that $f'(x) > 0$ for $0<x<\sqrt{\frac{4}{\pi}-1}$ and $f'(x) < 0$ for $\sqrt{\frac{4}{\pi}-1} < x < 1$. These two pieces of information show that $f$ is positive for all $0<x<1$ and has a positive maximum at $x =\sqrt{\frac{4}{\pi}-1}$.
\subsection{Proof of Equation \eqref{lambda1}}
    We similarly take $y =1/|\lambda_1|$, and thus $0\leq y \leq 1$. Let $0 < p \leq 1$. We show that
    \[
    f(y) = \frac{Cy^p}{(1+y^2)^\frac{p}{2}}-\arctan(y)
    \]
    satisfies Lemma \ref{MVT} for some $C(p)$ to be chosen. We see that $f(0) = 0$ and choose $C$ so that
    \[
    f(1) = \frac{C}{2^\frac{p}{2}}-\frac{\pi}{4} > 0 \quad \text{i.e.}\quad C > \frac{2^\frac{p}{2}\pi}{4}.
    \]
    We then calculate
    \begin{align*}
        f'(y) &= \frac{pCy^{p-1}}{(1+y^2)^\frac{p}{2}} - \frac{pCy^{p+1}}{(1+y^2)^{\frac{p}{2}+1}} - \frac{1}{1+y^2}\\
        &= \frac{pCy^p}{(1+y^2)^\frac{p}{2}}\left(\frac{1}{y}-\frac{y}{1+y^2}\right) - \frac{1}{1+y^2}\\
        &= \frac{pCy^p}{(1+y^2)^\frac{p}{2}}\frac{1}{y(1+y^2)} - \frac{1}{1+y^2}\\
        &=\frac{1}{1+y^2}\left(\frac{pC}{y^{1-p}(1+y^2)^\frac{p}{2}} - 1 \right).
    \end{align*}
    The function 
    \[
    g(y) = \frac{pC}{y^{1-p}(1+y^2)^\frac{p}{2}}
    \]
    is strictly decreasing for $y>0$. To show $f'(y)>0$, it suffices to show that $g(1) > 1$ and we see that
    \[
    1 <  \frac{pC}{2^\frac{p}{2}} =g(1) \quad\text{requires}\quad C > \frac{2^\frac{p}{2}}{p}.
    \]
    Combining this with the above condition on $C$, we can take
    \[
    C = \frac{2^\frac{p}{2}\pi}{p}.
    \]
    Lastly, we get that
    \begin{align*}
        \arctan(\lambda_1) &< \frac{C}{(1+\lambda_1^2)^\frac{p}{2}} \leq \frac{C}{(1 +\lambda_2^2)^\frac{p}{4}(1+\lambda_1^2)^\frac{p}{4}}\\
        &= \frac{C}{(\sqrt{(1+\lambda_1^2)(1+\lambda_2^2)})^\frac{p}{8}}\\
        &\leq \frac{C}{\log(\sqrt{(1+\lambda_1^2)(1+\lambda_2^2)})} = \frac{C}{b}\\
        &\leq \frac{C}{\bar{b}}\leq \frac{C}{\rho^2\bar{b}^{1-q}}.
    \end{align*}

\bibliographystyle{amsalpha}
\bibliography{version2}

\end{document}